\documentclass[12pt,reqno,openbib,runningheads,a4paper,portrait]{amsart}%
\usepackage{lineno}
\usepackage[authoryear,sectionbib]{natbib}
\usepackage[bookmarks=true,colorlinks=true,linkcolor=blue,citecolor=blue,a4paper]%
{hyperref}
\usepackage[labelformat=simple,labelfont={bf},textfont={up},labelsep=period]%
{caption}
\usepackage{amsfonts}
\usepackage{amsmath}
\usepackage{amssymb}
\usepackage{graphicx}%
\providecommand{\U}[1]{\protect\rule{.1in}{.1in}}
\providecommand{\U}[1]{\protect\rule{.1in}{.1in}}
\newtheorem{theorem}{Theorem}[section]

\newtheorem{proposition}{Proposition}[section]

\newtheorem{lemma}{Lemma}[section]

\newtheorem{remark}{Remark}[section]

\makeatletter
\renewcommand{\@biblabel}[1]{}
\makeatother
\begin{document}

\begin{center}
{\Large \textbf{Kernel estimation for the tail index of a right-censored
Pareto-type distribution}}\medskip\medskip

{\large Abdelhakim Necir}$^{\ast},$ {\large Louiza Soltane}\medskip\\[0pt]

{\small \textit{Department of Mathematics, Mohamed Khider University, Biskra,
Algeria}}\medskip
\end{center}

\noindent\textbf{Abstract}\medskip

\noindent We introduce a kernel estimator, to the tail index of a
right-censored Pareto-type distribution, that generalizes Worms's one
\citep{WW2014} in terms of weight coefficients. Under some regularity
conditions, the asymptotic normality of the proposed estimator is established.
In the framework of the second-order condition, we derive an asymptotically
bias-reduced version to the new estimator. Through a simulation study, we
conclude that one of the main features of the proposed kernel estimator is its
smoothness contrary to Worms's one, which behaves, rather erratically, as a
function of the number of largest extreme values. As expected, the bias
significantly decreases compared to that of the non-smoothed estimator with
however a slight increase in the mean squared error.\medskip\medskip

 \medskip


\vfill

\vfill

\noindent{\small $^{\text{*}}$Corresponding author:
\texttt{ah.necir@univ-biskra.dz} \newline\noindent\textit{E-mail
address:}\newline l.soltane@univ-biskra.dz (L.~Soltane)\newline}

\section{\textbf{Introduction\label{sec1}}}

\subsection{\textbf{A review of the tail index estimation}}

\noindent Let $X_{1},X_{2},...,X_{n}$ be independent and identically
distributed (iid) of non-negative random variables (rv's) as $n$ copies of a
rv $X,$ defined over some probability space $\left(  \Omega,\mathcal{A}%
,\mathbf{P}\right)  ,$ with cumulative distribution function (cdf) $F.\ $We
assume that the distribution tail $\overline{F}:=1-F$ is regularly varying at
infinity,\ with index $\left(  -1/\gamma_{1}\right)  ,$ notation:
$\overline{F}\in\mathcal{RV}_{\left(  -1/\gamma_{1}\right)  },$ that is%
\begin{equation}
\lim_{t\rightarrow\infty}\frac{\overline{F}\left(  tx\right)  }{\overline
{F}\left(  t\right)  }=x^{-1/\gamma_{1}},\text{ for any }x>0,
\label{first-condition}%
\end{equation}
where $\gamma_{1}>0$ is called the shape parameter or the tail index or the
extreme value index (EVI). It plays a very crucial role in the analysis of
extremes as it governs the thickness of the distribution right-tail. The most
popular estimator of $\gamma_{1}$ is Hill's estimator \cite{Hill75} defined by%
\begin{equation}
\widehat{\gamma}_{1,k}^{\left(  H\right)  }:=\frac{1}{k}%
{\displaystyle\sum\limits_{i=1}^{k}}
\log\frac{X_{n-i+1:n}}{X_{n-k:n}}=%
{\displaystyle\sum\limits_{i=1}^{k}}
\frac{i}{k}\log\frac{X_{n-i+1:n}}{X_{n-i:n}}, \label{Hill}%
\end{equation}
where $X_{1:n}\leq...\leq X_{n:n}$ denote the order statistics pertaining to
the sample $\left(  X_{1},...,X_{n}\right)  $ and $k=k_{n}$ is an integer
sequence satisfying $1<k<n,$ $k\rightarrow\infty$ and $k/n\rightarrow0$ as
$n\rightarrow\infty.$ The discrete character and non-stability of Hill's
estimator present major drawbacks. Indeed, adding a single large-order
statistic in the calculation of the estimator, that is, increasing $k$ by $1,$
may deviate from the true value of the estimate substantially. Thus, the
plotting of this estimator as a function of the upper order statistics often
gives a zig-zag figure (see Figure $\ref{Figure1}).$%

\begin{figure}
[ptb]
\begin{center}
\includegraphics[
trim=0.000000in 0.000000in -0.367615in 0.000000in,
height=3.0113in,
width=5.521in
]%
{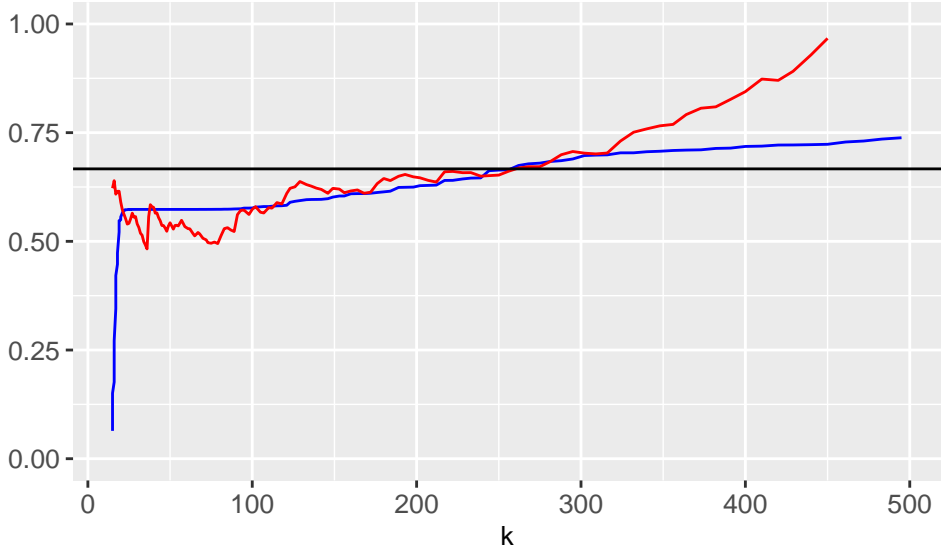}%
\caption{Ploting both Hill's (red line) and CDM's (blue line) tail index
estimators, as function of \noindent$k$ upper order statistics, for a
Pareto-type distribution.}%
\label{Figure1}%
\end{center}
\end{figure}

\noindent To overcome this issue, \cite{CDM85} introduced more general weighs
instead of the natural one $i/k$ that appears in the second formula of
$\widehat{\gamma}_{1,k}^{\left(  H\right)  },$ to define the following kernel
estimator
\[
\widehat{\gamma}_{1,k}^{\left(  CDM\right)  }\left(  K\right)  :=\sum
_{i=1}^{k}\frac{i}{k+1}K\left(  \frac{i}{k+1}\right)  \log\frac{X_{n-i+1:n}%
}{X_{n-i:n}},
\]
where $K$ is a kernel function satisfying the following assumptions:

\begin{itemize}
\item $\left[  A1\right]  $ is non increasing and right-continuous on
$\mathbb{R}.$

\item $\left[  A2\right]  $ $K\left(  s\right)  =0$ for $s\notin\left(
0,1\right]  $ and $K\left(  s\right)  \geq0$ for $s\in\left(  0,1\right]  .$

\item $\left[  A3\right]  $ $\int_{\mathbb{R}}K\left(  s\right)  ds=1.$

\item $\left[  A4\right]  $ $K$ and its first and second Lebesgue derivatives
$K^{\prime}$ and $K^{\prime\prime}$ are bounded.\medskip
\end{itemize}

\noindent The commonly used kernel functions are: the indicator kernel
$K_{1}:=\mathbf{1}\left\{  \left[  0,1\right)  \right\}  ,$ the biweight,
triweight and quadweight kernels respectively defined on $0\leq s<1$ by%
\begin{equation}
K_{2}\left(  s\right)  :=\frac{15}{8}\left(  1-s^{2}\right)  ^{2},\text{
}K_{3}\left(  s\right)  :=\frac{35}{16}\left(  1-s^{2}\right)  ^{3},\text{
}K_{4}\left(  s\right)  :=\frac{315}{128}\left(  1-s^{2}\right)  ^{4},
\label{Kfun}%
\end{equation}
and zero elsewhere, where $\mathbf{1}\left\{  A\right\}  $ stands for the
indicator function of a set $A.$ Note that the indicator kernel $K_{1}$
corresponds to the weigh coefficients of a closely related tail index
estimator to Hill's one $\widehat{\gamma}_{1,k}^{\left(  H\right)  }.$ The
nice properties of the kernel estimator $\widehat{\gamma}_{1,k}^{\left(
CDM\right)  }\left(  K\right)  $ are the smoothness and the stability,
contrary to Hill's one which rather exhibits fluctuations along the range of
upper extreme values. Thanks to these features, the exact choice of $k$ to be
used in the kernel estimator becomes not as crucial as that in Hill's one
\citep[see,
e.g.,][]{GLW2003}. For an overview of the kernel estimates of the tail index
for complete data, one refers to \cite{HLM2006}, \cite{CM2010}, \cite{GBW10}
and \cite{Caeiro19} and references therein. Motivated by the qualities of this
estimation method, recently \cite{BchMN2016} proposed a kernel estimator of
the tail index for randomly truncated data and established its asymptotic
normality. To the best of our knowledge, when the data are randomly censored,
this estimation approach is not yet addressed in the extreme value literature.
In the following section we present a review of the existing tail index
estimators and then propose kernel estimators to $\gamma_{1}$ for censored data.

\section{\textbf{Review of tail index estimation for censored data\label{sec2}%
}}

\noindent In the analysis of lifetime, reliability or insurance data, the
observations are usually randomly censored. In other words, in many real
situations the variable of interest $X$ is not always available. An
appropriate way to model this matter, is to introduce a non-negative rv $Y,$
called censoring rv, independent of $X$ and then to consider the rv
$Z:=\min\left(  X,Y\right)  $ and the indicator variable $\delta
:=\mathbf{1}\left\{  X\leq Y\right\}  ,$ which determines whether or not $X$
has been observed. The cdf's of $Y$ and $Z$ will be denoted by $G$ and $H$
respectively. The analysis of extreme values of randomly censored data is a
new research topic to which \cite{ReTo7} made a very brief reference, in
Section 6.1, as a first step but with no asymptotic results. Considering
Hall's model \citep{Hall82}, \cite{BGDF07} proposed estimators for the EVI and
high quantiles and discussed their asymptotic properties, when the data are
censored by a deterministic threshold. \cite{EFG2008} adapted various EVI
estimators to the case where data are censored by a random threshold and
proposed a unified method to establish their asymptotic normality. In this
context, the censoring distribution is assumed to be regularly varying too,
that is $\overline{G}\in\mathcal{RV}_{\left(  -1/\gamma_{2}\right)  },$ for
some $\gamma_{2}>0.$\ By virtue of the independence of $X$ and $Y,$ we have
$\overline{H}\left(  x\right)  =\overline{F}\left(  x\right)  \overline
{G}\left(  x\right)  $ and therefore $\overline{H}\in\mathcal{RV}_{\left(
-1/\gamma\right)  },$\ with $\gamma:=\gamma_{1}\gamma_{2}/\left(  \gamma
_{1}+\gamma_{2}\right)  .$ We also assume that both $\overline{F}$ and
$\overline{G}$ satisfy the second-order condition of regularly varying
functions:%
\begin{align}
\overline{F}\left(  x\right)   &  =C_{1}x^{-1/\gamma_{1}}\left(
1+D_{1}x^{-\beta_{1}}\left(  1+o\left(  1\right)  \right)  \right)  ,\text{ as
}x\rightarrow\infty,\label{C1}\\
\overline{G}\left(  y\right)   &  =C_{2}y^{-1/\gamma_{2}}\left(
1+D_{2}y^{-\beta_{2}}\left(  1+o\left(  1\right)  \right)  \right)  ,\text{ as
}y\rightarrow\infty, \label{C3}%
\end{align}
where $\beta_{1},\beta_{2},C_{1},C_{2}$ are positive constants and
$D_{1},D_{2}$ are real constants. The parameters $\tau_{i}:=-\beta_{i}%
\gamma_{i}<0,$ $i=1,2$ called the second-order parameters corresponding to
cdf's $F$ and $G$ respectively. This class of cdf's is known by Hall's models
which contains the most usual Pareto-type cdf's, namely Burr, Fr\'{e}chet,
GEV, GPD, Student, etc. The previous two conditions together imply that
\begin{equation}
\overline{H}\left(  z\right)  =Cz^{-1/\gamma}\left(  1+D_{\ast}z^{-\beta
_{\ast}\gamma}\left(  1+o\left(  1\right)  \right)  \right)  ,\text{ as
}z\rightarrow\infty, \label{H}%
\end{equation}
where $C:=C_{1}C_{2},$ $\beta_{\ast}:=\min\left(  \beta_{1},\beta_{2}\right)
$ and
\[
D_{\ast}:=D_{1}\mathbf{1}\left\{  \beta_{1}<\beta_{2}\right\}  +D_{2}%
\mathbf{1}\left\{  \beta_{1}>\beta_{2}\right\}  +\left(  D_{1}+D_{2}\right)
\mathbf{1}\left\{  \beta_{1}=\beta_{2}\right\}  .
\]
Let $\left\{  \left(  Z_{1},\delta_{1}\right)  ,...,\left(  Z_{n},\delta
_{n}\right)  \right\}  $ be a sample from the couple of rv's $\left(
Z,\delta\right)  $ and let $Z_{1:n}\leq...\leq Z_{n:n}$\ denote the order
statistics pertaining to $\left(  Z_{1},...,Z_{n}\right)  .$ If we denote the
concomitant of the $i$th order statistic by $\delta_{\left(  i\right)  }$
(i.e. $\delta_{\left(  i\right)  }=\delta_{j}$ if $Z_{i:n}=Z_{j}),$\ then the
adapted Hill estimator of the tail index $\gamma_{1}$ is defined by%
\begin{equation}
\widehat{\gamma}_{1,k}^{\left(  EFG\right)  }:=\frac{\widehat{\gamma}%
_{k}^{\left(  H\right)  }}{\widehat{p}_{k}}, \label{AH}%
\end{equation}
where $\widehat{\gamma}_{k}^{\left(  H\right)  }$ is Hill's estimator of the
tail index $\gamma$ and $\widehat{p}_{k}:=k^{-1}\sum\limits_{i=1}^{k}%
\delta_{\left(  n-i+1\right)  }$ denotes the estimator of asymptotic
proportion of non-censored observations in the tail given by
\begin{equation}
p:=\frac{\gamma}{\gamma_{1}}=\frac{\gamma_{2}}{\gamma_{1}+\gamma_{2}},
\label{p}%
\end{equation}
which is the limit, as $z\rightarrow\infty,$ of function%
\begin{equation}
p\left(  z\right)  :=\mathbf{P}\left(  \delta=1\mid Z=z\right)  . \label{pz}%
\end{equation}
Asymptotic representations both to $\widehat{\gamma}_{k}^{\left(  H\right)  }$
and\ $\widehat{p}_{k}$ in terms of Brownian bridges processes, given by
\cite{BMN15}, leading to the asymptotic normality of $\widehat{\gamma}%
_{1,k}^{\left(  EFG\right)  }$ under the usual second-order conditions of
regularly varying functions. In this context, the estimation of the
conditional tail index is addressed in \cite{Ndao14}, \cite{Ndao16},
\cite{Stup16} and \cite{Goeg2019}. Recently \cite{Stup19} assumed that the
censoring and the censored rv's are dependant and proposed an estimation
procedure to $\gamma_{1}.$\medskip

\noindent By using a Kaplan-Meier integral \cite{BWW2019} proposed in new
estimator to $\gamma_{1}$ defined by%
\begin{equation}
\widehat{\gamma}_{1,k}^{\left(  W\right)  }:=\sum_{j=2}^{k}\frac{\overline
{F}_{n}^{KM}\left(  Z_{n-j+1:n}\right)  }{\overline{F}_{n}^{KM}\left(
Z_{n-k:n}\right)  }\log\frac{Z_{n-j+1:n}}{Z_{n-j:n}}, \label{WW}%
\end{equation}
where%
\[
\overline{F}_{n}^{KM}\left(  t\right)  :=\left\{
\begin{array}
[c]{lc}%
\prod\limits_{i=1}^{n}\left(  1-\dfrac{\delta_{\left(  i\right)  }}%
{n-i+1}\right)  ^{\mathbf{1}\left\{  Z_{i:n}\leq t\right\}  } & \text{if
}t<Z_{n:n}\medskip\\
0 & \text{otherwise}%
\end{array}
\right.  ,
\]
denotes the well-known Kaplan-Meier product-limit estimator \citep{KM58} of
the underlying cdf $F.$ Actually $\widehat{\gamma}_{1,k}^{\left(  W\right)  }$
is a slight modification of the tail index estimator
\[
\widetilde{\gamma}_{1,k}^{\left(  W\right)  }:=\sum_{j=1}^{k}\frac
{\overline{F}_{n}^{KM}\left(  Z_{n-j:n}\right)  }{\overline{F}_{n}^{KM}\left(
Z_{n-k:n}\right)  }\log\frac{Z_{n-j+1:n}}{Z_{n-j:n}},
\]
first given by \cite{WW2014}. We showed in Proposition $\ref{Propo2}$ that the
increments
\[
\frac{\overline{F}_{n}^{KM}\left(  Z_{n-j:n}\right)  }{\overline{F}_{n}%
^{KM}\left(  Z_{n-k:n}\right)  }-\frac{\overline{F}_{n}^{KM}\left(
Z_{n-j+1:n}\right)  }{\overline{F}_{n}^{KM}\left(  Z_{n-k:n}\right)  },\text{
}j=1,...,k,
\]
are negligible (in probability) for all large $n,$ this means that the
estimation of $\overline{F}_{n}^{KM}\left(  Z_{n-j:n}\right)  /\overline
{F}_{n}^{KM}\left(  Z_{n-k:n}\right)  $ by $\overline{F}_{n}^{KM}\left(
Z_{n-j+1:n}\right)  /\overline{F}_{n}^{KM}\left(  Z_{n-k:n}\right)  $ is well
justified. The authors also showed that for a suitable sequence of integer
$k,$ we have
\begin{equation}
\sqrt{k}\left(  \widehat{\gamma}_{1,k}^{\left(  W\right)  }-\gamma_{1}\right)
\overset{\mathcal{D}}{\rightarrow}\mathcal{N}\left(  \lambda m,p\gamma_{1}%
^{2}/\left(  2p-1\right)  \right)  ,\text{ as }n\rightarrow\infty, \label{naw}%
\end{equation}
for some real constants $\lambda,$ where $m:=-\mathbf{1}\left\{  \beta_{1}%
\leq\beta_{2}\right\}  \gamma^{2}\beta_{1}D_{1}C^{-\gamma\beta_{1}}%
p^{-1}\left(  1+\beta_{1}\gamma/p\right)  ^{-1},$ provided that $p>1/2$ which
is equivalent to $\gamma_{1}<\gamma_{2}.$ In the real-life applications, the
latter assumption may be realizable. Indeed, \cite{BMV18} gave an application
to insurance data and claim that exhibits heavy-tail and part of these can be
considered to satisfy the $p>1/2.$ Through simulations, \cite{WW2014} showed
that $\widehat{\gamma}_{1,k}^{\left(  W\right)  }$ performs better than the
adapted Hill estimator $\widehat{\gamma}_{1,k}^{\left(  EFG\right)  },$ in the
weak-censoring case $p>1/2,$ both in term of bias and mean squared error
(MSE). However the estimator exhibits a slightly high bias which is natural
when one deals with Hill-type estimators. Instead of Kaplan-Meier approach,
\cite{BMN2016} proposed another asymptotically normal estimator of $\gamma
_{1}$ close to $\widehat{\gamma}_{1,k}^{\left(  W\right)  }$ that is based of
the Nelson-Aalen nonparametric estimator \citep{Nelson}, which seems to have a
slightly lower bias compared with that of $\widehat{\gamma}_{1,k}^{\left(
W\right)  }.$ As expected, the asymptotic biases and variances corresponding
to the two estimators meet.\medskip

\noindent Recently \cite{BAB21} proposed the following class of kernel
estimators defined by%
\begin{equation}
\widehat{\gamma}_{1,k}^{\left(  BAB\right)  }\left(  \mathcal{K}\right)
:=\frac{1}{k}\sum_{i=1}^{k}\mathcal{K}\left(  \frac{i}{k+1},\widehat{p}%
_{k}\right)  \frac{1}{\log\left(  \left(  k+1\right)  /i\right)  }\log
\frac{Z_{n-i+1:n}}{Z_{n-k:n}}, \label{BAB}%
\end{equation}
where $\widehat{p}_{k}$ is given in $\left(  \ref{AH}\right)  $ and
$\mathcal{K}$ is a positive kernel satisfying $p\int_{0}^{1}\mathcal{K}\left(
s,p\right)  du=1,$ for $p\in\left(  0,1\right]  .$ The particular kernel
functions used by the authors are:
\begin{equation}
\mathcal{K}_{0}\left(  s,p\right)  :=\frac{1}{p}\log\frac{1}{s},\text{
}\mathcal{K}_{1}\left(  s,p\right)  :=s^{p-1},\text{ }\mathcal{K}_{2}\left(
s,p\right)  :=\frac{s^{p-1}-1}{1-p}. \label{Kbis}%
\end{equation}
Since $\widehat{\gamma}_{1,k}^{\left(  BAB\right)  }\left(  \mathcal{K}%
_{0}\right)  \equiv\widehat{\gamma}_{1,k}^{\left(  EFG\right)  },$ then this
kernel estimator can be viewed as a generalization of$\ \widehat{\gamma}%
_{1,k}^{\left(  EFG\right)  },$ in terms of weight coefficients. As mentioned
in their paper, Worms's estimator $\widehat{\gamma}_{1,k}^{\left(  W\right)  }
$ does not fall into this framework, but it simplified version $\widehat
{\gamma}_{1,k}^{\left(  BAB\right)  }\left(  \mathcal{K}\right)  $ does. In
their simulation study, \cite{BAB21} pointed out that the MSE characteristics
of the estimator $\widehat{\gamma}_{1,k}^{\left(  BAB\right)  }\left(
\mathcal{K}_{2}\right)  $ are quite comparable to those of $\widehat{\gamma
}_{1,k}^{\left(  W\right)  }.$ Overall, however, $\widehat{\gamma}%
_{1,k}^{\left(  W\right)  }$ performs better $\widehat{\gamma}_{1,k}^{\left(
BAB\right)  }\left(  \mathcal{K}\right)  $ in terms of bias for the three
kernel functions. The asymptotic normality of this estimator is established by
considering both weak and strong censoring cases $\left(  \text{i.e. }%
0<p\leq1\right)  .$ We can summarize the features of BAB's estimator in two
points: its smoothness compared with $\widehat{\gamma}_{1,k}^{\left(
EFG\right)  }$ and its asymptotic normality which is hold for all $0<p\leq1$
however that of Worms's one is limited only to the interval $p>1/2.$
$\medskip$

\noindent In the following section we introduce a new kernel estimator for the
tail index $\gamma_{1}$ that generalizes $\widehat{\gamma}_{1,k}^{\left(
W\right)  }$ is the sense that the two estimators coincide for the indicator
function $K_{1}.$ In other terms, this new kernel estimator is a
generalization of CDM's estimator $\widehat{\gamma}_{1,k}^{\left(  CDM\right)
}\left(  K\right)  $ to the case of censored data.

\subsection{A new kernel estimator for $\gamma_{1}$}

\noindent By using Potter's inequalities, see e.g. Proposition B.1.10 in
\cite{deHF06}, to the regularly varying function $F$ together with assumptions
$[A1]-[A4],$ \cite{BchMN2016} showed that%
\[
\lim_{u\rightarrow\infty}\int_{u}^{\infty}g_{K}^{\prime}\left(  \frac
{\overline{F}\left(  x\right)  }{\overline{F}\left(  u\right)  }\right)
\log\frac{x}{u}d\frac{F\left(  x\right)  }{\overline{F}\left(  u\right)
}=\gamma_{1}\int_{0}^{\infty}K\left(  x\right)  dx=\gamma_{1},
\]
where $g_{K}\left(  x\right)  :=xK\left(  x\right)  $ and $g^{\prime}$ denotes
the Lebesgue derivative of $g.$ Since $\overline{F}$ is continuous, then
$\overline{F}\left(  x\right)  =\overline{F}\left(  x^{-}\right)  ,$ this
allows us to write
\[
\frac{\overline{F}\left(  x\right)  }{\overline{F}\left(  u\right)  }%
=\theta_{x,u}\frac{\overline{F}\left(  x^{-}\right)  }{\overline{F}\left(
u\right)  }+\left(  1-\theta_{x,u}\right)  \frac{\overline{F}\left(  x\right)
}{\overline{F}\left(  u\right)  },\text{ for }x>u,
\]
for some arbitrary real number $0<\theta_{x,u}<1.$ Next we will see that,
thanks to the mean value theorem, this formula provides us an estimation of
the derivative $g^{\prime}$ in terrmes of an increment of function $g;$ see
equation $\left(  \ref{mv}\right)  $ below. The notation $\psi\left(
a^{-}\right)  :=\lim_{x\uparrow a}\psi\left(  x\right)  $ stands for the
left-limit of a function $\psi\left(  t\right)  $ as $t$ approaches $a$ from
the left. By letting $u=Z_{n-k:n}$ and substituting $F$ by Kaplan-Meier
estimator $F_{n}^{KM},$ we derive a kernel estimator to the tail index
$\gamma_{1}$ defined by
\[
\widetilde{\gamma}_{1,k}\left(  K\right)  :=\int_{Z_{n-k:n}}^{\infty}%
g_{K}^{\prime}\left(  \mathcal{F}_{n}\left(  x\right)  \right)  \log\left(
\frac{x}{Z_{n-k:n}}\right)  \frac{dF_{n}^{KM}\left(  x\right)  }{\overline
{F}_{n}^{KM}\left(  Z_{n-k:n}\right)  },
\]
where%
\[
\mathcal{F}_{n}\left(  x\right)  :=\theta_{x,n}\frac{\overline{F}_{n}%
^{KM}\left(  x^{-}\right)  }{\overline{F}_{n}^{KM}\left(  Z_{n-k:n}\right)
}+\left(  1-\theta_{x,n}\right)  \frac{\overline{F}_{n}^{KM}\left(  x\right)
}{\overline{F}_{n}^{KM}\left(  Z_{n-k:n}\right)  },
\]
and $\theta_{x,n}:=\theta_{x,Z_{n-k:n}}$ $\left(  \text{arbitrary}\right)  .$
To rewrite the previous integral into a sum form, we use the following crucial
equation: for a given functional $\phi\left(  \cdot;F\right)  ,$ we have
\begin{equation}
\int\phi\left(  x;F\right)  dF\left(  x\right)  =\int\frac{\phi\left(
z;F\right)  }{\overline{G}\left(  z\right)  }\left\{  \overline{G}\left(
z\right)  dF\left(  z\right)  \right\}  =\int\frac{\phi\left(  z;F\right)
}{\overline{G}\left(  z\right)  }dH^{1}\left(  z\right)  , \label{eq}%
\end{equation}
where $H^{1}\left(  z\right)  :=\mathbf{P}\left(  Z\leq z,\delta=1\right)
=\int_{0}^{z}\overline{G}\left(  x\right)  dF\left(  x\right)  ,$ see for
instance \cite{Stute95}. The empirical counterparts of integrals in $\left(
\ref{eq}\right)  $ are
\[
\int\phi\left(  x;F_{n}^{KM}\right)  dF_{n}^{KM}\left(  x\right)  =\int
\frac{\phi\left(  z;F_{n}^{KM}\right)  }{\overline{G}_{n}^{KM}\left(
z^{-}\right)  }dH_{n}^{1}\left(  z\right)  ,
\]
where
\[
\overline{G}_{n}^{KM}\left(  t\right)  :=\left\{
\begin{array}
[c]{lc}%
{\displaystyle\prod\limits_{i=1}^{n}}
\left(  1-\dfrac{1-\delta_{\left(  i\right)  }}{n-i+1}\right)  ^{\mathbf{1}%
\left\{  Z_{i:n}\leq t\right\}  } & \text{if }t<Z_{n:n}\medskip\\
0 & \text{otherwise}%
\end{array}
\right.  ,
\]
denotes the Kaplan-Meire estimator of cdf $G$ and $H_{n}^{1}\left(  z\right)
:=n^{-1}\sum\limits_{i=1}^{n}\mathbf{1}\left\{  Z_{i}\leq z,\text{ }\delta
_{i}=1\right\}  ,$ is the empirical counterpart of the subdisribution function
$H^{1}.$ We used $\overline{G}_{n}^{KM}\left(  t^{-}\right)  $ instead of
$\overline{G}_{n}^{KM}\left(  t\right)  $ to avoid a division by zero, besides
that $\overline{G}_{n}^{KM}\left(  Z_{n:n}\right)  =0.$ For convenience, we
set
\[
\phi_{n}\left(  x;F\right)  :=\frac{1}{\overline{F}_{n}^{KM}\left(
Z_{n-k:n}\right)  }g_{K}^{\prime}\left(  \mathcal{F}_{n}\left(  x\right)
\right)  \log\frac{x}{Z_{n-k:n}}\mathbf{1}\left\{  x>Z_{n-k:n}\right\}  .
\]
Thus, the kernel estimator $\widetilde{\gamma}_{1,k}\left(  K\right)  $ may be
rewritten into%
\[
\int_{0}^{\infty}\phi_{n}\left(  x;F_{n}^{KM}\right)  dF_{n}\left(  x\right)
=\int_{0}^{\infty}\frac{\phi_{n}\left(  z;F_{n}^{KM}\right)  }{\overline
{G}_{n}^{KM}\left(  z^{-}\right)  }dH_{n}^{1}\left(  z\right)  ,
\]
which equals%
\begin{align*}
&  \sum_{i=1}^{n}\frac{\mathbf{1}\left\{  Z_{i:n}>Z_{n-k:n}\right\}
}{\overline{F}_{n}^{KM}\left(  Z_{n-k:n}\right)  }\frac{\delta_{\left(
i\right)  }}{n\overline{G}_{n}^{KM}\left(  Z_{i:n}^{-}\right)  }g_{K}^{\prime
}\left(  \mathcal{F}_{n}\left(  Z_{i:n}\right)  \right)  \log\frac{Z_{i:n}%
}{Z_{n-k:n}}.\\
&  =\sum_{i=n-k+1}^{n}\frac{\delta_{\left(  i\right)  }}{n\overline{F}%
_{n}^{KM}\left(  Z_{n-k:n}\right)  \overline{G}_{n}^{KM}\left(  Z_{i-1:n}%
\right)  }g_{K}^{\prime}\left(  \mathcal{F}_{n}\left(  Z_{i:n}\right)
\right)  \log\frac{Z_{i:n}}{Z_{n-k:n}},
\end{align*}
where%
\[
\mathcal{F}_{n}\left(  Z_{i:n}\right)  :=\theta_{i,n}\frac{\overline{F}%
_{n}^{KM}\left(  Z_{i-1:n}\right)  }{\overline{F}_{n}^{KM}\left(
Z_{n-k:n}\right)  }+\left(  1-\theta_{i,n}\right)  \frac{\overline{F}_{n}%
^{KM}\left(  Z_{i:n}\right)  }{\overline{F}_{n}^{KM}\left(  Z_{n-k:n}\right)
},
\]
where $\left(  \theta_{i,n}\right)  _{1\leq i\leq k}$ is an arbitrary random
sequence. By changing the index of summation $i$ to $n-j+1,$ yields%
\begin{equation}
\widetilde{\gamma}_{1,k}\left(  K\right)  =\sum_{j=1}^{k}\frac{\delta_{\left(
n-j+1\right)  }}{n\overline{F}_{n}^{KM}\left(  Z_{n-k:n}\right)  \overline
{G}_{n}^{KM}\left(  Z_{n-j:n}\right)  }g_{K}^{\prime}\left(  \mathcal{F}%
_{n}\left(  Z_{n-j+1:n}\right)  \right)  \log\frac{Z_{n-j+1:n}}{Z_{n-k:n}%
}.\nonumber
\end{equation}
We showed in Proposition $\ref{Propo1}$ that%
\[
\frac{\delta_{\left(  n-j+1\right)  }}{n\overline{G}_{n}^{KM}\left(
Z_{n-j:n}\right)  }=\overline{F}_{n}^{KM}\left(  Z_{n-j:n}\right)
-\overline{F}_{n}^{KM}\left(  Z_{n-j+1:n}\right)  ,
\]
therefore%
\begin{align*}
\widetilde{\gamma}_{1,k}\left(  K\right)   &  =\sum_{j=1}^{k}\left\{
\frac{\overline{F}_{n}^{KM}\left(  Z_{n-j:n}\right)  }{\overline{F}_{n}%
^{KM}\left(  Z_{n-k:n}\right)  }-\frac{\overline{F}_{n}^{KM}\left(
Z_{n-j+1:n}\right)  }{\overline{F}_{n}^{KM}\left(  Z_{n-k:n}\right)  }\right\}
\\
&  \ \ \ \ \ \ \ \ \ \ \ \ \ \ \times g_{K}^{\prime}\left(  \mathcal{F}%
_{n}\left(  Z_{n-j+1:n}\right)  \right)  \log\frac{Z_{n-j+1:n}}{Z_{n-k:n}}.
\end{align*}
In view of the mean value theorem, we may choose the sequence of constants
$\theta_{j,n}$ so that%
\begin{equation}
g_{K}^{\prime}\left(  \mathcal{F}_{n}\left(  Z_{n-j+1:n}\right)  \right)
=\frac{g_{K}\left(  \frac{\overline{F}_{n}^{KM}\left(  Z_{n-j:n}\right)
}{\overline{F}_{n}^{KM}\left(  Z_{n-k:n}\right)  }\right)  -g_{K}\left(
\frac{\overline{F}_{n}^{KM}\left(  Z_{n-j+1:n}\right)  }{\overline{F}_{n}%
^{KM}\left(  Z_{n-k:n}\right)  }\right)  }{\frac{\overline{F}_{n}^{KM}\left(
Z_{n-j:n}\right)  }{\overline{F}_{n}^{KM}\left(  Z_{n-k:n}\right)  }%
-\frac{\overline{F}_{n}^{KM}\left(  Z_{n-j+1:n}\right)  }{\overline{F}%
_{n}^{KM}\left(  Z_{n-k:n}\right)  }}, \label{mv}%
\end{equation}
thus
\begin{equation}
\widetilde{\gamma}_{1,k}\left(  K\right)  =\sum_{j=1}^{k}\left\{  g_{K}\left(
\frac{\overline{F}_{n}^{KM}\left(  Z_{n-j:n}\right)  }{\overline{F}_{n}%
^{KM}\left(  Z_{n-k:n}\right)  }\right)  -g_{K}\left(  \frac{\overline{F}%
_{n}^{KM}\left(  Z_{n-j+1:n}\right)  }{\overline{F}_{n}^{KM}\left(
Z_{n-k:n}\right)  }\right)  \right\}  \log\frac{Z_{n-j+1:n}}{Z_{n-k:n}}.
\label{gtild}%
\end{equation}
Recall that $g_{K}\left(  x\right)  =xK\left(  x\right)  $ and let
\[
a_{j}:=g_{K}\left(  \frac{\overline{F}_{n}^{KM}\left(  Z_{n-j:n}\right)
}{\overline{F}_{n}^{KM}\left(  Z_{n-k:n}\right)  }\right)  \text{ and }%
b_{j}=\log\frac{Z_{n-j:n}}{Z_{n-k:n}}.
\]
By applying Proposition $\ref{Propo3},$ we may rewrite formula $\left(
\ref{gtild}\right)  $ into%
\begin{equation}
\widetilde{\gamma}_{1,k}\left(  K\right)  =\sum_{j=1}^{k}\frac{\overline
{F}_{n}^{KM}\left(  Z_{n-j:n}\right)  }{\overline{F}_{n}^{KM}\left(
Z_{n-k:n}\right)  }K\left(  \frac{\overline{F}_{n}^{KM}\left(  Z_{n-j:n}%
\right)  }{\overline{F}_{n}^{KM}\left(  Z_{n-k:n}\right)  }\right)  \log
\frac{Z_{n-j+1:n}}{Z_{n-j:n}}. \label{K}%
\end{equation}
By using the same modification as made, in \cite{BWW2019}, to the original
formula of Worms's estimator $\widetilde{\gamma}_{1,k}^{\left(  W\right)  },$
that is substituting $\overline{F}_{n}^{KM}\left(  Z_{n-j:n}\right)
/\overline{F}_{n}^{KM}\left(  Z_{n-k:n}\right)  $ by $\overline{F}_{n}%
^{KM}\left(  Z_{n-j+1:n}\right)  /\overline{F}_{n}^{KM}\left(  Z_{n-k:n}%
\right)  ,$ we end up to the final form of our new kernel estimator given by%
\begin{equation}
\widehat{\gamma}_{1,k}\left(  K\right)  :=\sum_{j=2}^{k}\frac{\overline{F}%
_{n}^{KM}\left(  Z_{n-j+1:n}\right)  }{\overline{F}_{n}^{KM}\left(
Z_{n-k:n}\right)  }K\left(  \frac{\overline{F}_{n}^{KM}\left(  Z_{n-j+1:n}%
\right)  }{\overline{F}_{n}^{KM}\left(  Z_{n-k:n}\right)  }\right)  \log
\frac{Z_{n-j+1:n}}{Z_{n-j:n}}. \label{KF}%
\end{equation}
It is obvious that $\widehat{\gamma}_{1,k}\left(  K_{1}\right)  $ coincides
with Worms's estimator $\widehat{\gamma}_{1,k}^{\left(  W\right)  }$ stated in
$\left(  \ref{WW}\right)  .$ For the sake of simplicity, from now on where
there is no conflict, we limit ourselves to writing $\widehat{\gamma}_{1,k}$
instead of $\widehat{\gamma}_{1,k}\left(  K\right)  .$ Finally, by using the
bias-reduction approach given by \cite{BWW2019}, we derive an asymptotically
bias-reduced estimator corresponding to $\widehat{\gamma}_{1,k}$ defined by
\begin{equation}
\widehat{\gamma}_{1,k}^{\ast}:=\widehat{\gamma}_{1,k}-\widehat{\rho}\left\{
T_{k}\left(  -\widehat{\tau}_{1}/\widehat{\gamma}_{1,k};K\right)
-\widehat{\gamma}_{1,k}\widehat{\eta}_{2}\right\}  , \label{g1-hate}%
\end{equation}
where $\widehat{\tau}_{1}:=-\beta_{1}\widehat{\gamma}_{1,k}$ is a consistent
estimator of the second-order parameter $\tau_{1}:=-\beta_{1}\gamma_{1}$ of
cdf $F$ in $\left(  \ref{C1}\right)  ,$
\begin{equation}
\widehat{\rho}:=\frac{1}{\widehat{\eta}_{3}/\widehat{\eta}_{2}-\widehat{\eta
}_{1}}, \label{alpha}%
\end{equation}%
\[
\widehat{\eta}_{1}:=\int_{0}^{1}s^{-\widehat{\tau}_{1}}\left(  1-\widehat
{\tau}_{1}\log s\right)  K\left(  s\right)  ds,
\]%
\[
\widehat{\eta}_{2}:=\int_{0}^{1}s^{-\widehat{\tau}_{1}}K\left(  s\right)
ds,\text{ }\widehat{\eta}_{3}:=\int_{0}^{1}\left(  s^{-\widehat{\tau}_{1}%
}-s^{-2\widehat{\tau}_{1}}\right)  K\left(  s\right)  ds
\]
and
\begin{equation}%
\begin{array}
[c]{l}%
T_{k}\left(  \omega;K\right)  :=\dfrac{1}{\omega}%
{\displaystyle\sum\limits_{j=2}^{k}}
\dfrac{\overline{F}_{n}^{KM}\left(  Z_{n-j+1:n}\right)  }{\overline{F}%
_{n}^{KM}\left(  Z_{n-k:n}\right)  }K\left(  \dfrac{\overline{F}_{n}%
^{KM}\left(  Z_{n-j+1:n}\right)  }{\overline{F}_{n}^{KM}\left(  Z_{n-k:n}%
\right)  }\right)  \medskip\\
\ \ \ \ \ \ \ \ \ \ \ \ \ \ \ \ \ \ \ \ \ \ \ \ \ \ \times\left\{  \left(
\dfrac{Z_{n-j:n}}{Z_{n-k:n}}\right)  ^{-\omega}-\left(  \dfrac{Z_{n-j+1:n}%
}{Z_{n-k:n}}\right)  ^{-\omega}\right\}  ,\text{ }\omega>0.
\end{array}
\label{T-omega}%
\end{equation}
We checked when one substitute $K$ by the indicator kernel function $K_{1},$
$\widehat{\tau}_{1}$ by $-\beta_{1}\widehat{\gamma}_{1,k}$ and then
$\widehat{p}$ by $\widehat{\gamma}/\widehat{\gamma}_{1},$ the kernel
reduced-bias estimator $\widehat{\gamma}_{1,k}^{\ast}$ meets that of Worms's
one stated in \cite{BWW2019} (equation $\left(  9\right)  ).$ \medskip\ 

\noindent To the best of our knowledge, there is no estimator for $\tau_{1}, $
however there is an adaptive estimation method proposed by \cite{BMV18}, which
is based on the minimization of the sample variance to the corresponding
bias-reduced estimator of $\gamma_{1}.$ This adaptive estimator is defined by
$\widehat{\tau}_{1}:=\arg\min_{\tau_{1}\in\mathcal{A}}\sum_{k=2}^{n}\left(
\widehat{\gamma}_{1,k}^{\ast}-\overline{\widehat{\gamma}}_{1}^{\ast}\right)
^{2},$ where $\overline{\widehat{\gamma}}_{1}^{\ast}:=n^{-1}\sum_{k=2}%
^{n}\widehat{\gamma}_{1,k}^{\ast}$ and $\mathcal{A}:=\left\{
-0.5-0.1i\right\}  _{0\leq i\leq25}.$ The rest of the paper is organized as
follows. In Section 2, we present our main result, namely the asymptotic
normality both of $\widehat{\gamma}_{1,k}$ and $\widehat{\gamma}_{1,k}^{\ast}$
whose proofs are postponed to Section 4. The finite sample behavior of the
proposed estimators is checked by simulation in Section 3, where a comparison
with the already existing ones is made as well. Finally, some instrumental
Propositions and Lemmas are stated in the Appendix.

\section{\textbf{Main results\label{sec3}}}

\begin{theorem}
\textbf{\label{Theorem1}}Assume that both second-order conditions $\left(
\ref{C1}\right)  $ and $\left(  \ref{C3}\right)  $ hold. Let $k=k_{n}$ be a
sequence of integer such that $\sqrt{k}\left(  k/n\right)  ^{\gamma\beta
_{\ast}}\rightarrow\lambda,$ and if $\lambda=0$ that $n=O\left(  n^{B}\right)
$ for sufficiently large $B>0.$ For a given kernel function $K$ satisfying
assumptions $\left[  A1\right]  -\left[  A4\right]  ,$ we have $\sqrt
{k}\left(  \widehat{\gamma}_{1,k}-\gamma_{1}\right)  \overset{\mathcal{D}%
}{\rightarrow}\mathcal{N}\left(  \lambda m_{K},\sigma_{K}^{2}\right)  ,$ as
$n\rightarrow\infty,$ provided that $p>1/2,$ where $\sigma_{K}^{2}:=\gamma
_{1}^{2}\int_{0}^{1}s^{-1/p+1}K^{2}\left(  s\right)  ds$ and%
\begin{equation}
m_{K}:=-\mathbf{1}\left\{  \beta_{1}\leq\beta_{2}\right\}  \beta_{1}%
D_{1}C^{-\gamma\beta_{1}}\gamma_{1}^{2}\int_{0}^{1}s^{\beta_{1}\gamma_{1}%
}K\left(  s\right)  ds. \label{mk}%
\end{equation}

\end{theorem}

\begin{remark}
It is clear that $\sigma_{K_{1}}^{2}=p\gamma_{1}^{2}/\left(  2p-1\right)  $
and
\[
m_{K_{1}}=-\mathbf{1}\left\{  \beta_{1}\leq\beta_{2}\right\}  \gamma^{2}%
\beta_{1}D_{1}C^{-\gamma\beta_{1}}p^{-1}\left(  1+\beta_{1}\gamma/p\right)
^{-1},
\]
which coincide respectively with the asymptotic variance and the asymptotic
mean, $\lambda m,$ of Worm's estimator $\widehat{\gamma}_{1,k}^{\left(
W\right)  }$ stated in $\left(  \ref{naw}\right)  .$
\end{remark}

\noindent For the asymptotic normality of the kernel reduced-bias estimator
$\widehat{\gamma}_{1,k}^{\ast},$ we introduce to the following additional
notations:%
\begin{equation}
\eta_{1}\left(  \tau_{1}\right)  :=\int_{0}^{1}s^{-\tau_{1}}\left(  1-\tau
_{1}\log s\right)  K\left(  s\right)  ds, \label{eta1}%
\end{equation}%
\begin{equation}
\eta_{2}\left(  \tau_{1}\right)  :=\int_{0}^{1}s^{-\tau_{1}}K\left(  s\right)
ds,\text{ }\eta_{3}\left(  \tau_{1}\right)  :=\int_{0}^{1}\left(  s^{-\tau
_{1}}-s^{-2\tau_{1}}\right)  K\left(  s\right)  ds \label{eta12}%
\end{equation}
and
\begin{equation}
\rho\left(  \tau_{1}\right)  :=\left(  \eta_{3}\left(  \tau_{1}\right)
/\eta_{2}\left(  \tau_{1}\right)  -\eta_{1}\left(  \tau_{1}\right)  \right)
^{-1}. \label{rho-t}%
\end{equation}
It is worth mentioning that by these new notations, we have $\eta_{i}\left(
\widehat{\tau}_{1}\right)  \equiv\widehat{\eta}_{i},$ $i=1,2,3$ and
$\rho\left(  \widehat{\tau}_{1}\right)  \equiv\widehat{\rho}.$

\begin{theorem}
\textbf{\label{Theorem2}}Assume that the assumptions of Theorem
$\ref{Theorem1}$ hold, then%
\[
\sqrt{k}\left(  \widehat{\gamma}_{1,k}^{\ast}-\gamma_{1}\right)
\overset{\mathcal{D}}{\rightarrow}\mathcal{N}\left(  0,\sigma_{K}^{\ast
2}\right)  ,\text{ as }n\rightarrow\infty,
\]
provided that $p>1/2,$ where $\sigma_{K}^{\ast2}:=p\gamma_{1}^{2}\int_{0}%
^{1}t^{-1/p+1}\left(  \left(  1+\eta_{1}\rho\right)  -\rho s^{-\tau_{1}%
}\right)  ^{2}K^{2}\left(  t\right)  dt,$ with $\eta_{i}:=\eta_{i}\left(
\tau_{1}\right)  ,$ $i=1,2,3$ and $\rho:=\rho\left(  \tau_{1}\right)  .$
\end{theorem}

\subsection{\textbf{Discussion on the asymptotic biases and variances of
}$\widehat{\gamma}_{1,k}$\textbf{\ and} $\widehat{\gamma}_{1,k}^{\left(
W\right)  }$}

By considering three kernel functions $K_{2},K_{3}$ and $K_{4}$ introduced in
$\left(  \ref{Kfun}\right)  ,$ we show that the absolute asymptotic bias of
$\widehat{\gamma}_{1,k}$ is less than that of $\widehat{\gamma}_{1,k}^{\left(
W\right)  },$ however the asymptotic variance behaves opposite. In the other
terms $\left\vert m_{K}\right\vert <\left\vert m\right\vert $ and $\sigma
_{K}^{2}>\sigma^{2}.$ Indeed, let us write%
\[
\frac{\left\vert m_{K}\right\vert }{\left\vert m\right\vert }=\left(
1+\beta_{1}\gamma_{1}\right)  \int_{0}^{1}s^{\beta_{1}\gamma_{1}}K\left(
s\right)  ds=\left(  1+t\right)  \int_{0}^{1}s^{t}K\left(  s\right)
ds=:g\left(  t\right)  ,
\]
where $t=\beta_{1}\gamma_{1}>0.$ It is clear from Figure $\ref{Figure2}$ that
$g\left(  t\right)  <1,$ for any $t>0,$ and therefore $\left\vert
m_{K}\right\vert <\left\vert m\right\vert .$ To compare the two variances, let
us write%

\[
\frac{\sigma_{K}^{2}}{\sigma^{2}}=\left(  \frac{p}{2p-1}\right)  ^{-1}\int
_{0}^{1}s^{-1/p+1}K^{2}\left(  s\right)  ds=:h\left(  p\right)  ,
\]
for $p>1/2.$ The Figure $\ref{Figure3}$ shows in turns that $h\left(
p\right)  >1$ for any $p>1/2,$ which implies that $\sigma_{K}^{2}>\sigma^{2}.$
From the two figures we point out that the quadweight kernel provides a better
asymptotic bias compared with other ones, however the asymptotic variance of
its corresponding tail index estimator is the biggest one. Then for a
bais-variance trade-off, we suggest using the triweight kernel function.%

\begin{figure}
[ptb]
\begin{center}
\includegraphics[
height=2.4111in,
width=5.521in
]%
{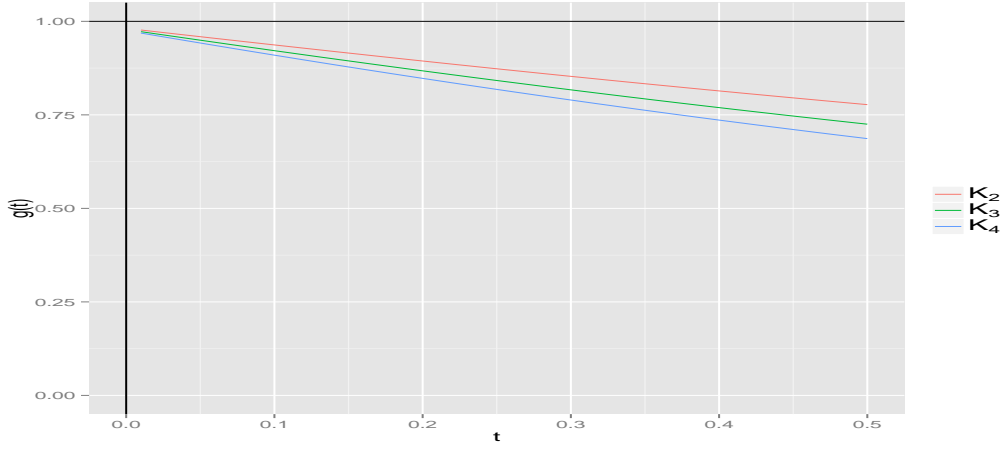}%
\caption{Plotting the function $g\left(  t\right)  $ over $t>0,$ for each
kernel $K_{2},K_{3}$ and $K_{4}.$}%
\label{Figure2}%
\end{center}
\end{figure}
%

\begin{figure}
[ptb]
\begin{center}
\includegraphics[
height=2.3903in,
width=5.521in
]%
{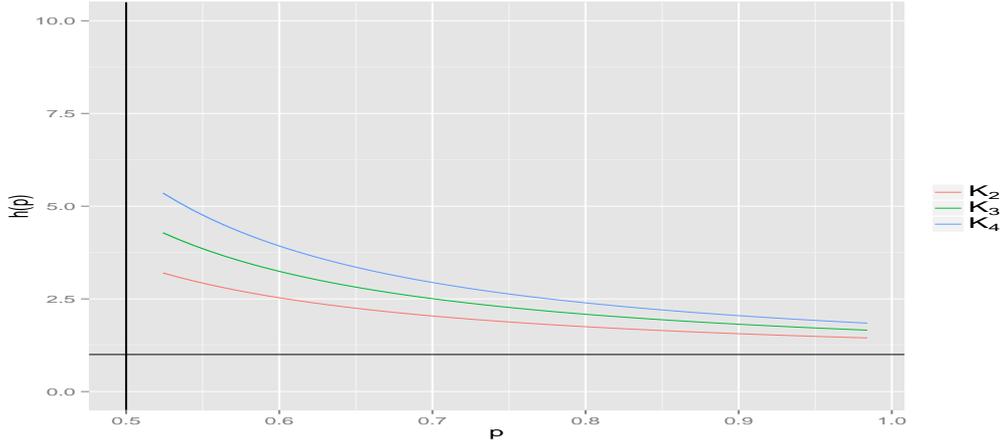}%
\caption{Plotting the function $h\left(  p\right)  $ over $\left(
1/2,1\right)  ,$ for each kernel $K_{2},K_{3}$ and $K_{4}.$}%
\label{Figure3}%
\end{center}
\end{figure}

\subsection{The optimal number of upper extremes}

For a given kernel $K,$ we seek the optimal number of upper extremes
$k_{K}^{\ast}$ that minimizes the asymptotic MSE which equals $\sigma_{K}%
^{2}/k+\left(  k/n\right)  ^{2\gamma\beta_{\ast}}m_{K}^{2}=:\mathcal{M}\left(
k\right)  .$ Explicitly we have%
\begin{align*}
&  \mathcal{M}\left(  k\right)
\begin{tabular}
[c]{l}%
$:=$%
\end{tabular}
\ \frac{\gamma_{1}^{2}}{k}\int_{0}^{1}s^{-1/p+1}K^{2}\left(  s\right)  ds\\
&  \ \ \ \ \ \ \ \ \ \ \ \ \ \ \ \ \ +\mathbf{1}\left\{  \beta_{1}\leq
\beta_{2}\right\}  \left(  k/n\right)  ^{2\gamma\beta_{1}}\beta_{1}^{2}%
D_{1}^{2}C^{-2\gamma\beta_{1}}\gamma_{1}^{4}\left\{  \int_{0}^{1}s^{\beta
_{1}\gamma_{1}}K\left(  s\right)  ds\right\}  ^{2}.
\end{align*}
Letting $\alpha:=\beta_{1}\gamma$ and $\mathcal{D}:=\mathbf{1}\left\{
\beta_{1}\leq\beta_{2}\right\}  D_{1}C^{-\alpha}p^{-1},$ we write%
\[
\gamma_{1}^{-2}\mathcal{M}\left(  k\right)  =\frac{1}{k}\int_{0}^{1}%
s^{-1/p+1}K^{2}\left(  s\right)  ds+\left(  \frac{k}{n}\right)  ^{2\alpha
}\left(  \alpha\mathcal{D}\right)  ^{2}\left\{  \int_{0}^{1}s^{\alpha
/p}K\left(  s\right)  ds\right\}  ^{2}.
\]
Using similar arguments as used to the proof of Theorem 5 in \cite{CDM85}, we
infer that $k_{K}^{\ast}$ minimizing the right-side of the previous equation
is the integer part of $n^{2\alpha/\left(  2\alpha+1\right)  }\left\{
2\alpha^{3}\mathcal{D}^{3}\right\}  ^{-1/\left(  2\alpha+1\right)  }%
\Phi\left(  K\right)  ,$ where%
\[
\Phi\left(  K\right)  :=\left\{  \int_{0}^{1}s^{-1/p+1}K^{2}\left(  s\right)
ds\right\}  ^{1/\left(  2\alpha+1\right)  }\left\{  \int_{0}^{1}s^{\alpha
/p}K\left(  s\right)  ds\right\}  ^{-2/\left(  2\alpha+1\right)  }.
\]
In particular, for the indicator kernel function $K_{1}:=\mathbf{1}\left\{
\left[  0,1\right)  \right\}  ,$ we have%
\begin{align*}
\Phi\left(  K_{1}\right)   &  =\left\{  \int_{0}^{1}s^{-1/p+1}ds\right\}
^{1/\left(  2\alpha+1\right)  }\left\{  \int_{0}^{1}s^{\alpha/p}ds\right\}
^{-2/\left(  2\alpha+1\right)  }\\
&  =\left(  \frac{p}{2p-1}\right)  ^{1/\left(  2\alpha+1\right)  }\left(
\alpha/p+1\right)  ^{2/\left(  2\alpha+1\right)  }.
\end{align*}
Thereby the optimal top $k$ observations used in Worms's estimator
$\widehat{\gamma}_{1,k}^{\left(  W\right)  }$ is
\begin{equation}
k_{W}^{\ast}=\left[  n^{2/\left(  2\alpha+1\right)  }\left(  \frac
{2p\alpha^{3}\mathcal{D}^{3}}{2p-1}\right)  ^{1/\left(  2\alpha+1\right)
}\left(  \alpha/p+1\right)  ^{2/\left(  2\alpha+1\right)  }\right]
.\label{k-W}%
\end{equation}
Thus the ratio between the two optimal number of extremes is
\begin{equation}
\frac{k_{K}^{\ast}}{k_{W}^{\ast}}\sim\frac{\left\{  \int_{0}^{1}%
s^{-1/p+1}K^{2}\left(  s\right)  ds\right\}  ^{1/\left(  2\alpha+1\right)
}\left\{  \int_{0}^{1}s^{\alpha/p}K\left(  s\right)  ds\right\}  ^{-2/\left(
2\alpha+1\right)  }}{\left(  \frac{p}{2p-1}\right)  ^{1/\left(  2\alpha
+1\right)  }\left(  \alpha/p+1\right)  ^{2/\left(  2\alpha+1\right)  }%
}.\label{ratio}%
\end{equation}
Unfortunately, the optimal choice of the number of upper order statistics to
be used in estimation depends mainly on the unknown slowly varying part of the
tail. This fact makes obtaining a practical strategy for minimizing the
asymptotic mean square error through an appropriate choice of $k$ difficult.
There are numerous heuristic methods to select the optimal number of upper
extremes used in the computation of the tail index estimate. An exhaustive
bibliography to this topic is gathered in the nice survey given by
\cite{CG-15}. Our choice fell on the method of Reiss and Thomas given in
\cite{ReTo7}, page $137.$ In this procedure one defines the optimal sample
fraction%
\[
k^{\ast}:=\arg\min_{1<k<n}\frac{1}{k}\sum_{i=1}^{k}i^{\nu}\left\vert
\widehat{\gamma}_{1,i}-\text{median}\left\{  \widehat{\gamma}_{1,1}%
,...,\widehat{\gamma}_{1,k}\right\}  \right\vert ,
\]
with suitable constant $0\leq\nu\leq1/2,$ where $\widehat{\gamma}_{1,i}$
corresponds to the kernel estimator of tail index $\gamma_{1},$ based on the
$i$ upper order statistics, of a Pareto-type model. We claim, in our
simulation study below, that $\nu=0.3$ provides better results both in terms
of bias and MSE. This agrees with that was found by \cite{NA2004} when
considering Hill's estimator in the non-truncation case.\ We will use this
\ procedure to select $k^{\ast}$ the optimal numbers of upper order statistics
used in the computation of the all aforementioned estimators.

\section{\textbf{Simulation study\label{sec4}}}

\noindent In this section we will perform a simulation study in order to
compare the finite sample behavior of the kernel estimator $\widehat{\gamma
}_{1,k},$ given in $\left(  \ref{gtild}\right)  ,$ with the three estimators
$\widehat{\gamma}_{1,k}^{\left(  EFG\right)  },$ $\widehat{\gamma}%
_{1,k}^{\left(  W\right)  }$ and $\widehat{\gamma}_{1,k}^{\left(  BAB\right)
}\left(  \mathcal{K}\right)  $ stated respectively in $\left(  \ref{AH}%
\right)  , $ $\left(  \ref{WW}\right)  $ and $\left(  \ref{BAB}\right)  .$ We
constructed the two estimator $\widehat{\gamma}_{1,k}$ and $\widehat{\gamma
}_{1,k}^{\left(  BAB\right)  }\left(  \mathcal{K}\right)  $ by selecting the
triweight kernel function $K_{3}$ (defined in $\left(  \ref{Kfun}\right)  ),$
and $\mathcal{K}_{2}$ (given in in $\left(  \ref{Kbis}\right)  )$
respectively. For the censoring and censored distributions functions $F$ and
$G,$ will be chosen among the following two models:

\begin{itemize}
\item Burr $\left(  \zeta,\gamma\right)  $ distribution with right-tail
function:%
\[
\overline{\mathcal{L}}\left(  x\right)  =\left(  1+x^{1/\zeta}\right)
^{-\zeta/\gamma},\text{ }x\geq0,\text{ }\zeta>0,\text{ }\gamma>0.
\]

\item Fr\'{e}chet $\left(  \gamma\right)  $ distribution with right-tail
function:%
\[
\overline{\mathcal{L}}\left(  x\right)  =1-\exp\left(  -x^{-1/\gamma}\right)
,\text{ }x>0,\text{ }\gamma>0.
\]

\end{itemize}

\noindent For each given distribution, we generate $2000$ random samples of
length $n$ $=$ $500$ and plot the four estimators, their corresponding biases
and MSE's as function of $k=1,...,500.$ We consider four scenarios, namely: a
Burr distribution censored by another Burr distribution (Figure $\ref{Figure4}%
$) a Fr\'{e}chet distribution censored by another Fr\'{e}chet distribution
(Figure $\ref{Figure5}$), a Burr distribution censored by a Fr\'{e}chet
distribution (Figure $\ref{Figure6}$) and a Fr\'{e}chet distribution censored
by a Burr distribution (Figure $\ref{Figure7}$). In each scenario, we
considered the two censoring schemes, that is the weak censoring $\left(
p>1/2\right)  $ and the strong censoring $\left(  p<1/2\right)  .$ The
parametrization of Fr\'{e}chet and Burr models is made so that it covers both
the two situations $p>1/2$ and $p<1/2.$ In right panels of the four Figures
$\ref{Figure4}$-$\ref{Figure7},$ the simulation study shows that both kernel
estimators $\widehat{\gamma}_{1,k}$ and $\widehat{\gamma}_{1,k}^{\left(
BAB\right)  }$ present a smoothness contrary to both estimators $\widehat
{\gamma}_{1,k}^{\left(  W\right)  }$ and $\widehat{\gamma}_{1,k}^{\left(
EFG\right)  },$ which behave erratically a long the range of the largest
extreme values $k.$ In addition, the two kernel estimators exhibit a stability
and alignment with respect to the true value of the tail index $\gamma_{1}$
over almost the interval. We also point out that, in terms of stability,
$\widehat{\gamma}_{1,k}$ performs better than $\widehat{\gamma}_{1,k}^{\left(
BAB\right)  }$ in the strong censoring case. We notice that $\widehat{\gamma
}_{1,k}^{\left(  W\right)  }$ meets the true value of the tail index in a
single point $k^{\ast}\in\left\{  1,...,500\right\}  $ while $\widehat{\gamma
}_{1,k}^{\left(  EFG\right)  }$ does not cross the line $\gamma_{1}$ at any
point, but approaches slightly this one on a small interval of $k.$ From the
middle panels (resp. the right panels), in overall, $\widehat{\gamma}_{1,k}$
performs better than the three other estimators in terms of bias (resp. MSE)
for the strong censoring case $\left(  p>1/2\right)  ,$ however $\widehat
{\gamma}_{1,k}^{\left(  W\right)  }$ seems to be slightly better than
$\widehat{\gamma}_{1,k}$ for the weak censoring $\left(  p<1/2\right)  $ one.%

\begin{figure}
[ptbh]
\begin{center}
\includegraphics[
height=3.0113in,
width=5.521in
]%
{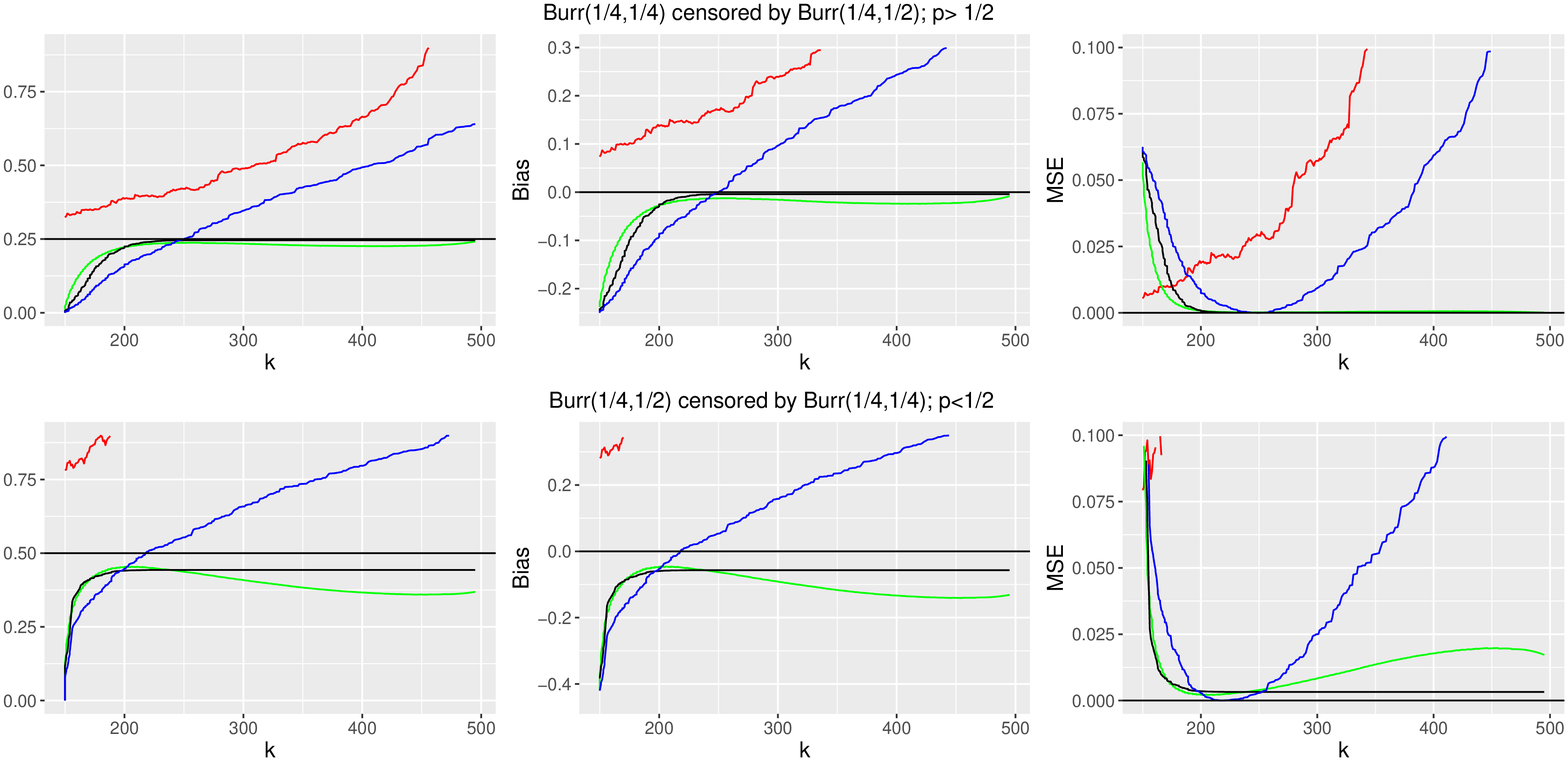}%
\caption{Comparison of the estimators (left-panels) $\widehat{\gamma}%
_{1,k}^{(EFG)}\left(  \text{red line}\right)  ,$ $\widehat{\gamma}_{1,k}%
^{(W)}\left(  \text{blue line}\right)  ,$ $\widehat{\gamma}_{1,k}%
^{(BAB)}\left(  \text{green line}\right)  ,$ $\widehat{\gamma}_{1,k}$ $\left(
\text{black line}\right)  $ their biases (middel-panels) and MSE's
(right-panels) for a Burr distribution censored by another Burr distribution
with $p=2/3$ (top-panels) and $p=1/3$ (bottom-panels)}%
\label{Figure4}%
\end{center}
\end{figure}
\begin{figure}
[ptbh]
\begin{center}
\includegraphics[
height=3.039in,
width=5.5486in
]%
{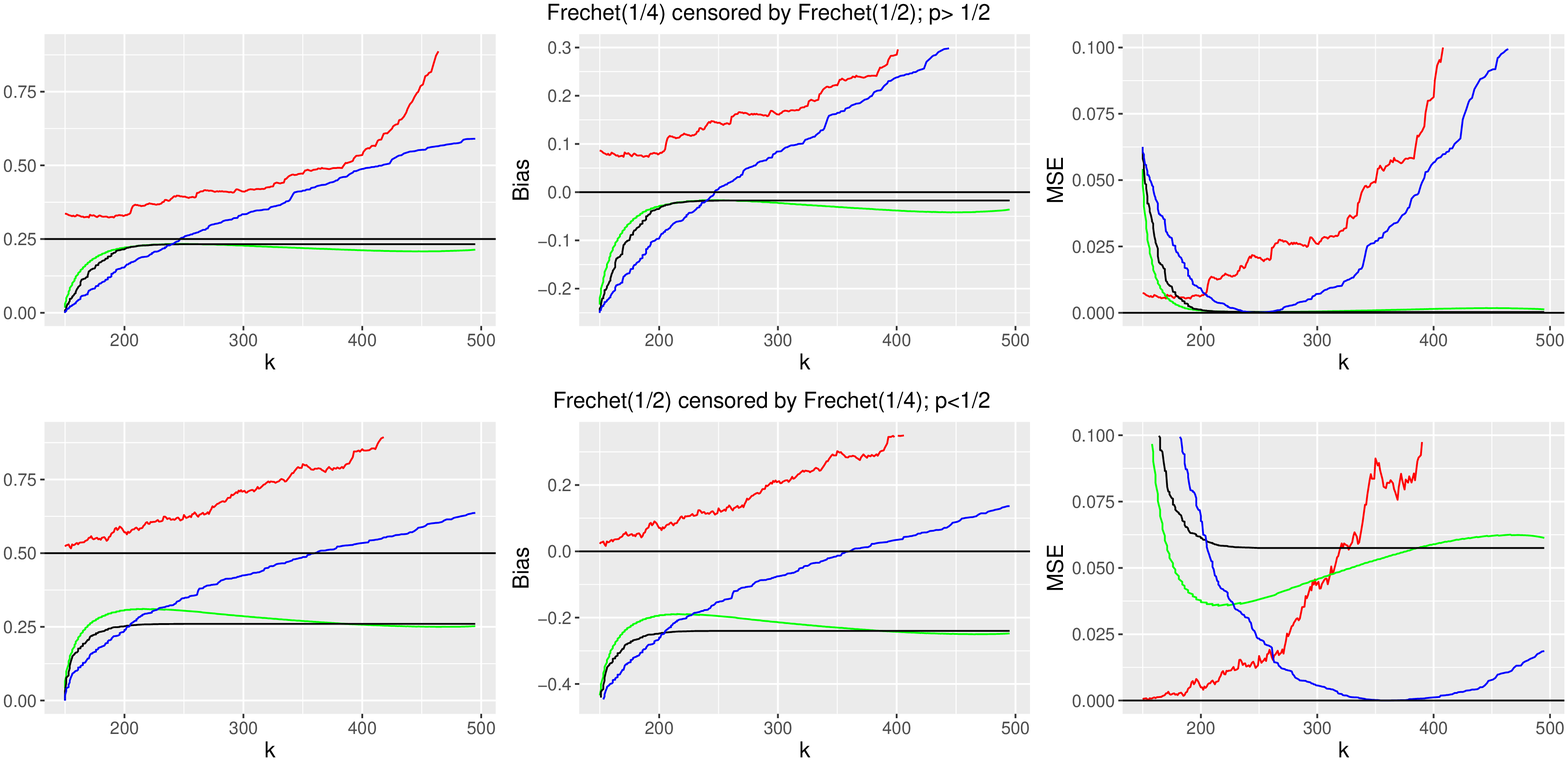}%
\caption{Comparison of the estimators (left-panels) $\widehat{\gamma}%
_{1,k}^{(EFG)}\left(  \text{red line}\right)  ,$ $\widehat{\gamma}_{1,k}%
^{(W)}\left(  \text{blue line}\right)  ,$ $\widehat{\gamma}_{1,k}%
^{(BAB)}\left(  \text{green line}\right)  ,$ $\widehat{\gamma}_{1,k}$ $\left(
\text{black line}\right)  $ their biases (middel-panels) and MSE's
(right-panels) for a Fr\'{e}chet distribution censored by another Fr\'{e}chet
distribution with $p=2/3$ (top-panels) and $p=1/3$ (bottom-panels)}%
\label{Figure5}%
\end{center}
\end{figure}
%

\begin{figure}
[ptbh]
\begin{center}
\includegraphics[
height=3.0113in,
width=5.521in
]%
{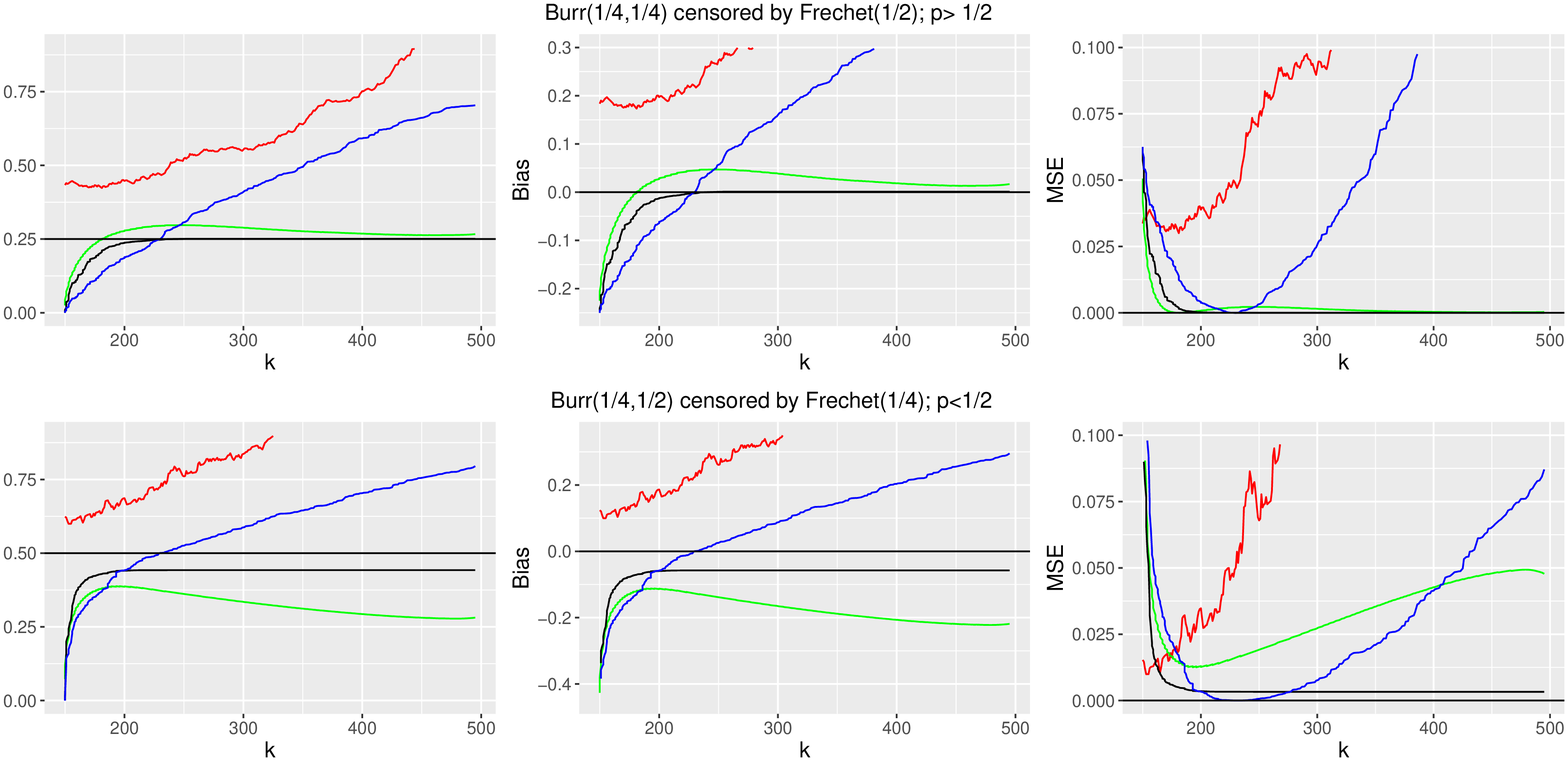}%
\caption{Comparison of the estimators (left-panels) $\widehat{\gamma}%
_{1,k}^{(EFG)}\left(  \text{red line}\right)  ,$ $\widehat{\gamma}_{1,k}%
^{(W)}\left(  \text{blue line}\right)  ,$ $\widehat{\gamma}_{1,k}%
^{(BAB)}\left(  \text{green line}\right)  ,$ $\widehat{\gamma}_{1,k}$ $\left(
\text{black line}\right)  $ their biases (middel-panels) and MSE's
(right-panels) for a Burr distribution censored by Fr\'{e}chet distribution
with $p=2/3$ (top-panels) and $p=1/3$ (bottom-panels)}%
\label{Figure6}%
\end{center}
\end{figure}
%

\begin{figure}
[ptbh]
\begin{center}
\includegraphics[
height=3.0113in,
width=5.521in
]%
{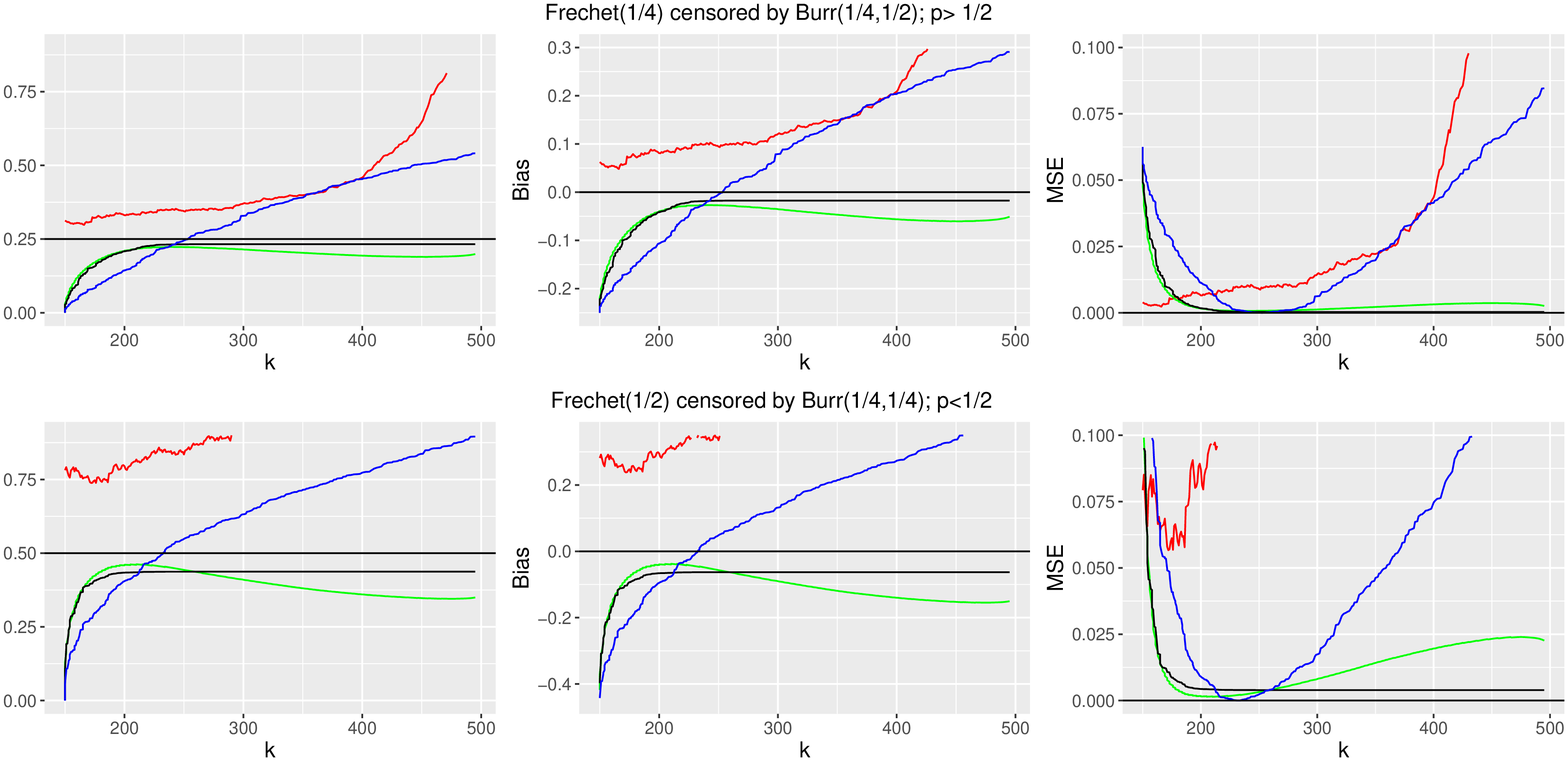}%
\caption{Comparison of the estimators (left-panels) $\widehat{\gamma}%
_{1,k}^{(EFG)}\left(  \text{red line}\right)  ,$ $\widehat{\gamma}_{1,k}%
^{(W)}\left(  \text{blue line}\right)  ,$ $\widehat{\gamma}_{1,k}%
^{(BAB)}\left(  \text{green line}\right)  ,$ $\widehat{\gamma}_{1,k}$ $\left(
\text{black line}\right)  $ their biases (middel-panels) and MSE's
(right-panels) for a Fr\'{e}chet distribution censored by Burr distribution
with $p=2/3$ (top-panels) and $p=1/3$ (bottom-panels)}%
\label{Figure7}%
\end{center}
\end{figure}

\section{\textbf{Proofs\label{sec5}}}

\subsection{Proof of Theorem $\ref{Theorem1}$}

We will adapt the proof of Theorem 1 in \cite{BWW2019} to the framework of the
kernel estimation. To begin, let us define the following quantities%
\[
\widehat{RF}_{j}:=\frac{\overline{F}_{n}^{KM}\left(  Z_{n-j+1:n}\right)
}{\overline{F}_{n}^{KM}\left(  Z_{n-k:n}\right)  },\text{ }RF_{j}%
:=\frac{\overline{F}\left(  Z_{n-j+1:n}\right)  }{\overline{F}\left(
Z_{n-k:n}\right)  },
\]%
\[
\widehat{RF}_{j}\left(  K\right)  :=\frac{\overline{F}_{n}^{KM}\left(
Z_{n-j+1:n}\right)  }{\overline{F}_{n}^{KM}\left(  Z_{n-k:n}\right)  }K\left(
\frac{\overline{F}_{n}^{KM}\left(  Z_{n-j+1:n}\right)  }{\overline{F}_{n}%
^{KM}\left(  Z_{n-k:n}\right)  }\right)  ,\text{ }%
\]
and%
\[
RF_{j}\left(  K\right)  :=\frac{\overline{F}\left(  Z_{n-j+1:n}\right)
}{\overline{F}\left(  Z_{n-k:n}\right)  }K\left(  \frac{\overline{F}\left(
Z_{n-j+1:n}\right)  }{\overline{F}\left(  Z_{n-k:n}\right)  }\right)  .
\]
For further use, we set
\[
\xi_{j}:=j\log\frac{Z_{n-j+1:n}}{Z_{n-j:n}}\text{ and }E_{j}^{\left(
n\right)  }:=j\log\frac{Y_{n-j+1:n}}{Y_{n-j:n}},\text{ for }j=1,...,k,
\]
where $Y_{1:n}\leq...\leq Y_{n:n}$ be the order statistics pertaining to the
sample $\left(  Y_{j}\right)  _{1\leq j\leq n}$ of iid standard Pareto rv's
defined by $Z_{j}=U_{H}\left(  Y_{j}\right)  ,$ with
\begin{equation}
U_{H}\left(  s\right)  :=\inf\left\{  x,\text{ }H\left(  x\right)
\geq1-1/s\right\}  ,\text{ }s>1, \label{UH1}%
\end{equation}
stands for the quantile function pertaining to cdf $H.$ Thanks to $\left(
\ref{H}\right)  ,$ we have%
\begin{equation}
U_{H}\left(  s\right)  =C^{\gamma}s^{\gamma}\left(  1+\gamma D_{\ast}%
C^{-\beta_{\ast}\gamma}s^{-\beta_{\ast}\gamma}\left(  1+o\left(  1\right)
\right)  \right)  ,\text{ as }s\rightarrow\infty. \label{UH}%
\end{equation}
Since $\left\{  \log Y_{j}\right\}  _{1\leq j\leq k}$ are iid standard
exponential rv's, then the normalized spacings $\left\{  E_{j}^{\left(
n\right)  }\right\}  _{1\leq j\leq k}$ are iid standard exponential rv's too;
see for instance Theorem 4.6.1 in \cite{Arnorld08}. The following
approximation, given by \cite{BDG2002}, will be one of the basic keys of the
proof:
\begin{equation}
\xi_{j}=\xi_{j}^{\prime}+R_{j,n}\text{ where }\xi_{j}^{\prime}:=\left(
\gamma+u_{j,k}^{\gamma\beta_{\ast}}b_{n,k}\right)  E_{j}^{\left(  n\right)  },
\label{ksiprime}%
\end{equation}
with $u_{j,k}:=\frac{j}{k+1}$ and $b_{n,k}=-\left(  1+o_{\mathbf{P}}\left(
1\right)  \right)  \gamma^{2}\beta_{\ast}D_{\ast}C^{-\gamma\beta_{\ast}%
}\left(  \frac{k+1}{n+1}\right)  ^{\gamma\beta_{\ast}},$ as $n\rightarrow
\infty.$ The remainder term $R_{j,n}$ is described in Theorem 2.1 of the
aforementioned paper, which satisfies $\left\vert \sum_{i=j}^{k}i^{-1}%
R_{i,n}\right\vert =o_{\mathbf{P}}\left(  b_{n,k}\log\left(  \max\left(
u_{j,k}^{-1},1\right)  \right)  \right)  .$ Next we show that, in the same
probability space $\left(  \Omega,\mathcal{A},\mathbf{P}\right)  ,$ there
exists a sequence of iid standard uniform rv's $\left(  U_{i}\right)  _{1\leq
i\leq k}$ independent to $\left(  E_{i}^{\left(  n\right)  }\right)  _{1\leq
i\leq k},$ such that%
\begin{align}
&  \sqrt{k}\left(  \widehat{\gamma}_{1,k}-\gamma_{1}\right)  -m_{K}\left\{
\sqrt{k}\left(  k/n\right)  ^{\gamma\beta_{\ast}}\right\} \nonumber\\
&  =\gamma\frac{\sqrt{k}}{k+1}\sum_{i=2}^{k}\left\{  u_{i,k}^{-1}\sum
_{j=2}^{i}u_{j,k}^{p-1}g_{K}^{\prime}\left(  u_{j,k}^{p}\right)  \right\}
\mathcal{A}_{i,n}+o_{\mathbf{P}}\left(  1\right)  , \label{final-app}%
\end{align}
where $g_{K}\left(  s\right)  =sK\left(  s\right)  ,$ $\mathcal{A}%
_{i,n}:=p\left(  E_{i}^{\left(  n\right)  }-1\right)  -\left(  \mathbf{1}%
\left\{  U_{i}\leq p\right\}  -p\right)  $ and $m_{K}$ is as in $\left(
\ref{mk}\right)  .$ To this end, let us rewrite formula $\left(
\ref{KF}\right)  $ into
\[
\widehat{\gamma}_{1,k}=\sum_{j=2}^{k}\frac{\overline{F}_{n}^{KM}\left(
Z_{n-j+1:n}\right)  }{\overline{F}_{n}^{KM}\left(  Z_{n-k:n}\right)  }K\left(
\frac{\overline{F}_{n}^{KM}\left(  Z_{n-j+1:n}\right)  }{\overline{F}_{n}%
^{KM}\left(  Z_{n-k:n}\right)  }\right)  \left\{  \log\frac{Z_{n-j+1:n}%
}{Z_{n-k:n}}-\log\frac{Z_{n-j:n}}{Z_{n-k:n}}\right\}  .
\]
It is easy to verify that $\widehat{\gamma}_{1,k}-\gamma_{1}$ may be
decomposed into the sum of%
\[
T_{k,n}^{\left(  1\right)  }:=\sum_{j=2}^{k}\left(  \widehat{RF}_{j}\left(
K\right)  -RF_{j}\left(  K\right)  \right)  \frac{\xi_{j}^{\prime}}{j},\text{
}T_{k,n}^{\left(  2\right)  }:=\sum_{j=2}^{k}RF_{j}\left(  K\right)  \frac
{\xi_{j}^{\prime}}{j}-\frac{\gamma}{k+1}\sum_{j=2}^{k}K\left(  u_{j,k}%
^{p}\right)  ,
\]%
\[
T_{k,n}^{\left(  3\right)  }:=\frac{\gamma}{k+1}\sum_{j=2}^{k}K\left(
u_{j,k}^{p}\right)  -\gamma_{1}\text{ and }T_{k,n}^{\left(  4\right)  }%
:=\sum_{j=2}^{k}\widehat{RF}_{j}\left(  K\right)  \frac{R_{j,n}}{j}.
\]
It is worth mentioning that, by considering the indicator kernel function
$K_{1},$ the last three (remainder) terms $T_{k,n}^{\left(  i\right)  },$
$i=2,3,4$ coincide with those stated in the beginning of the proof of Theorem
1 in \cite{BWW2019}. The authors showed that these terms, times $\sqrt{k},$
tend to zero in probability as $n\rightarrow\infty.$ By deep reading the
proof, we came to the conclusion that by using the assumption $\left[
A4\right]  $ on kernel $K$ we end up with $\sqrt{k}T_{k,n}^{\left(  i\right)
}=o_{\mathbf{P}}\left(  1\right)  ,$ $i=2,3,4$ as $n\rightarrow\infty,$ as
well, that we omit details. This means that $T_{k,n}^{\left(  1\right)  }$ is
the only term that contributes to the asymptotic normality of $\widehat
{\gamma}_{1,k}.$ Indeed, using Taylor's expansion of second-order to this one
yields%
\begin{align*}
T_{k,n}^{\left(  1\right)  }  &  =\sum_{j=2}^{k}\left(  \widehat{RF}%
_{j}-RF_{j}\right)  g_{K}^{\prime}\left(  RF_{j}\right)  \frac{\xi_{j}%
^{\prime}}{j}+\frac{1}{2}\sum_{j=2}^{k}\left(  \widehat{RF}_{j}-RF_{j}\right)
^{2}g_{K}^{\prime\prime}\left(  \widetilde{RF}_{j}\right)  \frac{\xi
_{j}^{\prime}}{j}\\
&  =\mathcal{T}_{k,n}^{\left(  1\right)  }+\mathcal{R}_{k,n}^{\left(
1\right)  },
\end{align*}
where $\widetilde{RF}_{j}$ is a rv between $\widehat{RF}_{j}$ and $RF_{j}.$
From assumption $\left[  A4\right]  ,$ the function $g_{K}^{\prime\prime}$ is
bounded, then using similar arguments of the proof in A.3.2 given in
\cite{BWW2019}, we show that $\sqrt{k}\mathcal{R}_{k,n}^{\left(  1\right)
}=o_{\mathbf{P}}\left(  1\right)  ,$ as $n\rightarrow\infty.$ Let us now focus
on the term $\mathcal{T}_{k,n}^{\left(  1\right)  }$ which may be made into
the sum of%
\[
\mathcal{T}_{k,n}^{\left(  1,1\right)  }:=\sum_{j=2}^{k}\left\{  \log
\frac{\widehat{RF}_{j}}{RF_{j}}\right\}  RF_{j}g_{K}^{\prime}\left(
RF_{j}\right)  \frac{\xi_{j}^{\prime}}{j}%
\]
and%
\[
\mathcal{T}_{k,n}^{\left(  1,2\right)  }:=\sum_{j=2}^{k}\left\{  -\log
\frac{\widehat{RF}_{j}}{RF_{j}}+\left(  \frac{\widehat{RF}_{j}}{RF_{j}%
}-1\right)  \right\}  RF_{j}g_{K}^{\prime}\left(  RF_{j}\right)  \frac{\xi
_{j}^{\prime}}{j}.
\]
It easy to check that $\log\widehat{RF}_{j}=\sum_{i=j}^{k}\delta_{\left(
n-i+1\right)  }\log\frac{i-1}{i},$ therefore%
\[
\mathcal{T}_{k,n}^{\left(  1,1\right)  }=\sum_{j=2}^{k}\left(  \sum_{i=j}%
^{k}\delta_{\left(  n-i+1\right)  }\log\frac{i-1}{i}-\log RF_{j}\right)
RF_{j}g_{K}^{\prime}\left(  RF_{j}\right)  \frac{\xi_{j}^{\prime}}{j}.
\]
Note that $\sum_{i=j}^{k}\xi_{i}/i=\log\left(  Z_{n-j+1:n}/Z_{n-k:n}\right)  $
and%
\[
\log RF_{j}=\frac{-1}{\gamma_{1}}\sum_{i=j}^{k}\frac{\xi_{i}}{i}+\left(  \log
RF_{j}+\frac{1}{\gamma_{1}}\log\frac{Z_{n-j+1:n}}{Z_{n-k:n}}\right)  ,
\]
\ it follows that%
\begin{align*}
\mathcal{T}_{k,n}^{\left(  1,1\right)  }  &  =\sum_{j=2}^{k}\left(  \frac
{1}{\gamma_{1}}\sum_{i=j}^{k}\frac{\xi_{i}}{i}+\sum_{i=j}^{k}\delta_{\left(
n-i+1\right)  }\log\frac{i-1}{i}\right)  RF_{j}g_{K}^{\prime}\left(
RF_{j}\right)  \frac{\xi_{j}^{\prime}}{j}\\
&  \ \ \ -\sum_{j=2}^{k}\left\{  \log RF_{j}+\frac{1}{\gamma_{1}}\log
\frac{Z_{n-j+1:n}}{Z_{n-k:n}}\right\}  RF_{j}g_{K}^{\prime}\left(
RF_{j}\right)  \frac{\xi_{j}^{\prime}}{j}.
\end{align*}
Observe that the first term equals%
\[
\sum_{j=2}^{k}\sum_{i=j}^{k}\left(  \frac{1}{\gamma_{1}}\frac{\xi_{i}}%
{i}+\delta_{\left(  n-i+1\right)  }\log\frac{i-1}{i}\right)  RF_{j}%
g_{K}^{\prime}\left(  RF_{j}\right)  \frac{\xi_{j}^{\prime}}{j},
\]
which, by inverting the sums, becomes $\sum_{i=2}^{k}\left(  \frac{1}%
{\gamma_{1}}\xi_{i}+\delta_{\left(  n-i+1\right)  }i\log\frac{i-1}{i}\right)
S_{i,k},$ \ where $S_{i,k}:=\frac{1}{i}\sum_{j=2}^{i}RF_{j}g_{K}^{\prime
}\left(  RF_{j}\right)  \frac{\xi_{j}^{\prime}}{j},$ $j=2,...,k.$ Thereby
$\mathcal{T}_{k,n}^{\left(  1,1\right)  }$ may be rewritten into%
\[%
\begin{array}
[c]{cc}%
{\displaystyle\sum\limits_{i=2}^{k}}
\left(  \dfrac{1}{\gamma_{1}}\left(  \xi_{i}-\gamma\right)  +\delta_{\left(
n-i+1\right)  }i\log\dfrac{i-1}{i}+p\right)  S_{i,k}\medskip & \\
-%
{\displaystyle\sum\limits_{j=2}^{k}}
\left(  \log RF_{j}+\dfrac{1}{\gamma_{1}}\log\dfrac{Z_{n-j+1:n}}{Z_{n-k:n}%
}\right)  RF_{j}g_{K}^{\prime}\left(  RF_{j}\right)  \dfrac{\xi_{j}^{\prime}%
}{j} & =\mathcal{T}_{k,n}^{\left(  1,1,1\right)  }-\mathcal{T}_{k,n}^{\left(
1,1,2\right)  }.
\end{array}
\]
We need to the following additional notations:%
\[
c_{i}:=1+i\log\frac{i-1}{i},\text{ }A_{i,n}:=p\left(  E_{i}^{\left(  n\right)
}-1\right)  -\left(  \delta_{\left(  n-i+1\right)  }-p\right)  \text{ and
}B_{i,n}:=\frac{1}{\gamma_{1}}b_{n,k}u_{i,k}^{\beta_{\ast}\gamma}%
E_{i}^{\left(  n\right)  }.
\]
By adding $\delta_{\left(  n-i+1\right)  }$ and subtracting it, then by using
the approximation $\left(  \ref{ksiprime}\right)  ,$ we rewrite $\mathcal{T}%
_{k,n}^{\left(  1,1,1\right)  }$ into
\[
\mathcal{T}_{k,n}^{\left(  1,1,1\right)  }=\sum_{i=2}^{k}A_{i,n}S_{i,k}%
+\sum_{i=2}^{k}B_{i,n}S_{i,k}+\sum_{i=2}^{k}\delta_{\left(  n-i+1\right)
}c_{i}S_{i,k}+\gamma_{1}^{-1}\sum_{i=2}^{k}R_{i,n}S_{i,k}.
\]
Once again, using assumption $\left[  A4\right]  $ and Proposition (parts c
and d) in \cite{BWW2019}, we infer that $\sqrt{k}\sum_{i=2}^{k}\delta_{\left(
n-i+1\right)  }c_{i}S_{i,k}=o_{\mathbf{P}}\left(  1\right)  =\sqrt{k}%
\sum_{i=2}^{k}R_{i,n}S_{i,k},$ as $n\rightarrow\infty.$ Next we show that%
\begin{equation}
\sqrt{k}\sum_{i=2}^{k}A_{i,n}S_{i,k}=\mathcal{N}\left(  0,\gamma_{1}^{2}%
\int_{0}^{1}t^{-1/p+1}K^{2}\left(  t\right)  dt\right)  +\mathcal{B}%
_{1,k}+o_{\mathbf{p}}\left(  1\right)  , \label{clt}%
\end{equation}
and $\sqrt{k}\sum_{i=2}^{k}B_{i,n}S_{i,k}=\mathcal{B}_{2,k}+o_{\mathbf{P}%
}\left(  1\right)  ,$ where $\mathcal{B}_{1,k}$ and $\mathcal{B}_{2,k}$ are
asymptotic two biases that we precise later on. Indeed, let us decompose
$\sum_{i=2}^{k}A_{i,n}S_{i,k}$ into the sum of%
\[
\mathcal{I}_{n1}:=\sum_{i=2}^{k}A_{i,n}\left\{  \frac{1}{i}\sum_{j=2}%
^{i}u_{j,k}^{p}g_{K}^{\prime}\left(  u_{j,k}^{p}\right)  \frac{\xi_{j}%
^{\prime}}{j}\right\}  ,
\]%
\[
\mathcal{I}_{n2}:=\sum_{i=2}^{k}A_{i,n}\left\{  \frac{1}{i}\sum_{j=2}%
^{i}\left\{  V_{j,k}^{p}g_{K}^{\prime}\left(  V_{j,k}^{p}\right)  -u_{j,k}%
^{p}g_{K}^{\prime}\left(  u_{j,k}^{p}\right)  \right\}  \frac{\xi_{j}^{\prime
}}{j}\right\}
\]
and%
\[
\mathcal{I}_{n3}:=\sum_{i=2}^{k}A_{i,n}\left\{  \frac{1}{i}\sum_{j=2}%
^{i}V_{j,k}^{p}g_{K}^{\prime}\left(  V_{j,k}^{p}\right)  C_{j,k,0}\frac
{\xi_{j}^{\prime}}{j}\right\}  ,
\]
where $C_{j,k,0}$ is a sequence of constants defined in assertion $\left(
33\right)  $ in \cite{BWW2019}. By means of Taylor's expansion to function
$t\rightarrow tg_{K}^{\prime}\left(  t\right)  ,$ with assumption $\left[
A4\right]  ,$ and similar arguments as used to terms $I_{n2}$ and $I_{n3}$ in
\cite{BWW2019}, we infer that $\sqrt{k}\mathcal{I}_{n2}=o_{\mathbf{P}}\left(
1\right)  =\sqrt{k}\mathcal{I}_{n3}.$ Recall, from representation $\left(
\ref{ksiprime}\right)  ,$ that $\xi_{j}^{\prime}$ may be rewritten into
$\gamma+\gamma\left(  E_{j}^{\left(  n\right)  }-1\right)  +b_{n,k}%
u_{j,k}^{\gamma\beta_{\ast}}E_{j}^{\left(  n\right)  },$ this allows us to
decompose the first term $\mathcal{I}_{n1}$ into the sum of
\[
\mathcal{I}_{n1}^{\left(  1\right)  }:=\gamma\sum_{i=2}^{k}A_{i,n}\left\{
\frac{1}{i}\sum_{j=2}^{i}u_{j,k}^{p}g_{K}^{\prime}\left(  u_{j,k}^{p}\right)
\frac{1}{j}\right\}  .
\]%
\[
\mathcal{I}_{n1}^{\left(  2\right)  }:=\gamma\sum_{i=2}^{k}A_{i,n}\left\{
\frac{1}{i}\sum_{j=2}^{i}u_{j,k}^{p}g_{K}^{\prime}\left(  u_{j,k}^{p}\right)
\frac{E_{j}^{\left(  n\right)  }-1}{j}\right\}
\]
and%
\[
\mathcal{I}_{n1}^{\left(  3\right)  }:=b_{n,k}\sum_{i=2}^{k}A_{i,n}\left\{
\frac{1}{i}\sum_{j=2}^{i}u_{j,k}^{p}g_{K}^{\prime}\left(  u_{j,k}^{p}\right)
\frac{u_{j,k}^{\gamma\beta_{\ast}}}{j}E_{j}^{\left(  n\right)  }\right\}  .
\]
Since $g_{K}^{\prime}$ is bounded, then using similar arguments as used to the
terms $I_{1,n}^{\left(  2\right)  }$ and $I_{1,n}^{\left(  3\right)  }$ in
\cite{BWW2019}, we show that $\sqrt{k}\mathcal{I}_{n1}^{\left(  2\right)
}=o_{\mathbf{P}}\left(  1\right)  =\sqrt{k}\mathcal{I}_{n1}^{\left(  3\right)
}$ as well, that we omit the details. Let us now focus on the first term
\[
\mathcal{I}_{n1}^{\left(  1\right)  }=\frac{\gamma}{k+1}\sum_{i=2}^{k}\left\{
\frac{1}{i}\sum_{j=2}^{i}u_{j,k}^{p-1}g_{K}^{\prime}\left(  u_{j,k}%
^{p}\right)  \right\}  \left\{  p\left(  E_{i}^{\left(  n\right)  }-1\right)
-\left(  \delta_{\left(  n-i+1\right)  }-p\right)  \right\}  .
\]
\cite{EFG2008} showed that the sequence of rv's $\delta_{\left(  i\right)  }$
may be approximated by iid Bernoulli rv's $\mathbf{1}\left\{  U_{i}\leq
p\right\}  ,$ where $U_{i}$ is a sequence of standard uniform rv's which are
independent to $E_{i}^{\left(  n\right)  }.$ Moreover the authors claim that
\[
\delta_{\left(  i\right)  }=\mathbf{1}\left\{  U_{i}\leq p\left(
Z_{i:n}\right)  \right\}  =\mathbf{1}\left\{  U_{i}\leq p\circ U_{H}\left(
Y_{i:n}\right)  \right\}  ,
\]
where $p\left(  \cdot\right)  $ is the function defined in Section
$\ref{sec2}.$ In order to use this approximation, let us decompose
$\mathcal{I}_{n1}^{\left(  1\right)  }$ into the sum of%
\[
\mathcal{L}_{n1}^{\left(  1\right)  }:=\gamma\frac{\sqrt{k}}{k+1}\sum
_{i=2}^{k}\left\{  \frac{1}{i}\sum_{j=2}^{i}u_{j,k}^{p-1}g_{K}^{\prime}\left(
u_{j,k}^{p}\right)  \right\}  \left(  p\left(  E_{i}^{\left(  n\right)
}-1\right)  -\left(  \mathbf{1}\left\{  U_{n-i+1}\leq p\right\}  -p\right)
\right)  ,
\]%
\begin{align*}
&  \mathcal{L}_{n1}^{\left(  2\right)  }%
\begin{array}
[c]{c}%
:=
\end{array}
-\gamma\frac{\sqrt{k}}{k+1}\sum_{i=2}^{k}\left\{  \frac{1}{i}\sum_{j=2}%
^{i}u_{j,k}^{p-1}g_{K}^{\prime}\left(  u_{j,k}^{p}\right)  \right\} \\
&  \ \ \ \ \ \ \ \ \ \ \ \ \ \ \ \ \ \times\left(  \mathbf{1}\left\{
U_{n-i+1}\leq p\circ U_{H}\left(  n/i\right)  \right\}  -\mathbf{1}\left\{
U_{n-i+1}\leq p\right\}  \right)
\end{align*}
and%
\begin{align*}
&  \mathcal{L}_{n1}^{\left(  3\right)  }%
\begin{array}
[c]{c}%
:=
\end{array}
\gamma\frac{\sqrt{k}}{k+1}\sum_{i=2}^{k}\left\{  \frac{1}{i}\sum_{j=2}%
^{i}u_{j,k}^{p-1}g_{K}^{\prime}\left(  u_{j,k}^{p}\right)  \right\} \\
&  \ \ \ \ \ \ \ \ \ \ \ \ \ \ \times\left(  \mathbf{1}\left\{  U_{n-i+1}\leq
p\circ U_{H}\left(  n/i\right)  \right\}  -\mathbf{1}\left\{  U_{n-i+1}\leq
p\circ U_{H}\left(  Y_{n-i+1:n}\right)  \right\}  \right)  .
\end{align*}
Note that the symbol $f_{1}\circ f_{2}$ sands for the composition of two
functions $f_{1}$ and $f_{2}.$ Note that $\left(  U_{n-i+1}\right)  _{1\leq
i\leq n}\overset{\mathcal{D}}{=}\left(  U_{i}\right)  _{1\leq i\leq n}$ and
$\left(  E_{i}^{\left(  n\right)  }\right)  _{1\leq i\leq k}\overset
{\mathcal{D}}{=}\left(  E_{i}\right)  _{1\leq i\leq k},$ where $\left(
E_{i}^{\left(  n\right)  }\right)  _{1\leq i\leq k}$ is a sequence of iid
standard exponential rv's, then without loss of generality we may write
\[
\mathcal{L}_{n1}^{\left(  1\right)  }=\gamma\frac{\sqrt{k}}{k+1}\sum_{i=2}%
^{k}\left\{  \frac{1}{i}\sum_{j=2}^{i}u_{j,k}^{p-1}g_{K}^{\prime}\left(
u_{j,k}^{p}\right)  \right\}  \left(  p\left(  E_{i}-1\right)  -\left(
\mathbf{1}\left\{  U_{i}\leq p\right\}  -p\right)  \right)  .
\]
Observe that this last may be decomposed into the sum of
\[
\mathcal{L}_{n1}^{\left(  1,1\right)  }:=\gamma\frac{\sqrt{k}}{k+1}\sum
_{i=2}^{k}d_{i,k}\left(  g_{K}\right)  \left(  p\left(  E_{i}-1\right)
-\left(  \mathbf{1}\left\{  U_{i}\leq p\right\}  -p\right)  \right)  ,
\]
and%
\[
\mathcal{L}_{n1}^{\left(  1,2\right)  }:=\gamma\frac{\sqrt{k}}{k+1}\sum
_{i=2}^{k}\left\{  u_{i,k}^{-1}\int_{0}^{u_{i,k}}s^{p-1}g_{K}^{\prime}\left(
s^{p}\right)  ds\right\}  \left(  p\left(  E_{i}-1\right)  -\left(
\mathbf{1}\left\{  U_{i}\leq p\right\}  -p\right)  \right)  ,
\]
where $d_{i,k}\left(  g_{K}\right)  :=\frac{1}{i}\sum_{j=2}^{i}u_{j,k}%
^{p-1}g_{K}^{\prime}\left(  u_{j,k}^{p}\right)  -u_{i,k}^{-1}\int_{0}%
^{u_{i,k}}s^{p-1}g_{K}^{\prime}\left(  s^{p}\right)  ds.$ Making use of Lemma
$\ref{lemma1}$ (see the Appendix) and using similar arguments as used to the
term $R_{k,n}$ in the proof of Proposition $1$ (part (a)) in \cite{BWW2019},
we show that $\mathcal{L}_{n1}^{\left(  1,1\right)  }=o_{\mathbf{P}}\left(
1\right)  ,$ that we omit further details. It is clear that the variance of
$p\left(  E_{i}-1\right)  -\left(  \mathbf{1}\left\{  U_{i}\leq p\right\}
-p\right)  $ equals $p^{2}+p\left(  1-p\right)  =p,$ thus using Lyapunov's
central limit theorem (for triangular arrays), we get
\[
\mathcal{L}_{n1}^{\left(  1,2\right)  }\overset{\mathcal{D}}{\rightarrow
}\mathcal{N}\left(  0,p\gamma^{2}\int_{0}^{1}s^{-2}\left(  \int_{0}^{s}%
t^{p-1}g_{K}^{\prime}\left(  t^{p}\right)  dt\right)  ^{2}ds\right)  ,\text{
as }n\rightarrow\infty.
\]
By using a change of variables, we readily showed that%
\[
p\gamma^{2}\int_{0}^{1}s^{-2}\left(  \int_{0}^{s}t^{p-1}g_{K}^{\prime}\left(
t^{p}\right)  dt\right)  ^{2}ds=\gamma_{1}^{2}\int_{0}^{1}t^{-1/p+1}%
K^{2}\left(  t\right)  dt.
\]
Let us consider the second term $\mathcal{L}_{n1}^{\left(  2\right)  }.$ From
assertion $\left(  57\right)  $ in \cite{BWW2019}, we infer that $p\circ
U_{H}\left(  n/i\right)  -p=p\left(  1-p\right)  \left(  D_{\gamma}\right)
_{\ast}\beta_{\ast}C^{-\gamma\beta_{\ast}}\left(  i/n\right)  ^{\gamma
\beta_{\ast}}\left(  1+o\left(  1\right)  \right)  ,$ where%
\[
\left(  D_{\gamma}\right)  _{\ast}:=\gamma_{1}D_{1}\mathbf{1}\left\{
\beta_{1}<\beta_{2}\right\}  -\gamma_{2}D_{2}\mathbf{1}\left\{  \beta
_{1}>\beta_{2}\right\}  +\left(  \gamma_{1}D_{1}-\gamma_{2}D_{2}\right)
\mathbf{1}\left\{  \beta_{1}=\beta_{2}\right\}  .
\]
It follows that
\[
\mathcal{L}_{n1}^{\left(  2\right)  }=-\gamma p\left(  1-p\right)  \left(
D_{\gamma}\right)  _{\ast}\beta_{\ast}C^{-\gamma\beta_{\ast}}\frac{\sqrt{k}%
}{k+1}\sum_{i=2}^{k}\left\{  i^{-1}\sum_{j=2}^{i}u_{j,k}^{p-1}g_{K}^{\prime
}\left(  u_{j,k}^{p}\right)  \right\}  \left(  i/n\right)  ^{\gamma\beta
_{\ast}},
\]
which may be decomposed into the sum of%
\[
\mathcal{L}_{n1}^{\left(  2,1\right)  }:=-\gamma p\left(  1-p\right)  \left(
D_{\gamma}\right)  _{\ast}\beta_{\ast}C^{-\gamma\beta_{\ast}}\frac{\sqrt{k}%
}{k+1}\sum_{i=2}^{k}\left\{  u_{i,k}^{-1}\int_{0}^{u_{i,k}}s^{p-1}%
g_{K}^{\prime}\left(  s^{p}\right)  \right\}  \left(  i/n\right)
^{\gamma\beta_{\ast}}%
\]
and%
\[
\mathcal{L}_{n1}^{\left(  2,2\right)  }:=-\gamma p\left(  1-p\right)  \left(
D_{\gamma}\right)  _{\ast}\beta_{\ast}C^{-\gamma\beta_{\ast}}\frac{\sqrt{k}%
}{k+1}\sum_{i=2}^{k}d_{i,k}\left(  g_{K}\right)  \left(  i/n\right)
^{\gamma\beta_{\ast}}.
\]
Observe that $\mathcal{L}_{n1}^{\left(  2,1\right)  }=\left\{  \frac{1}%
{k+1}\sum_{i=2}^{k}\left(  u_{i,k}^{-1+\gamma\beta_{\ast}}\int_{0}^{u_{i,k}%
}s^{p-1}g_{K}^{\prime}\left(  s^{p}\right)  \right)  \right\}  b_{1,k},$ where
$b_{1,k}:=-\gamma p\left(  1-p\right)  \left(  D_{\gamma}\right)  _{\ast}%
\beta_{\ast}C^{-\gamma\beta_{\ast}}\sqrt{k}\left(  \frac{k+1}{n}\right)
^{\gamma\beta_{\ast}}.$ The quantity between two braces is a Riemann sum,
which converges, as $n\rightarrow\infty,$ to
\[
\int_{0}^{1}s^{-1+\gamma\beta_{\ast}}\left(  \int_{0}^{s}t^{p-1}g_{K}^{\prime
}\left(  t^{p}\right)  dt\right)  ds=\frac{1}{p^{2}}\int_{0}^{1}t^{\gamma
_{1}\beta_{\ast}}K\left(  t\right)  dt.
\]
Recall that $\sqrt{k}\left(  \frac{k}{n}\right)  ^{\gamma\beta_{\ast}%
}=O_{\mathbf{P}}\left(  1\right)  ,$ then $\sqrt{k}\left(  \frac{k+1}%
{n}\right)  ^{\gamma\beta_{\ast}}=\sqrt{k}\left(  \frac{k}{n}\right)
^{\gamma\beta_{\ast}}+o_{\mathbf{P}}\left(  1\right)  ,$ thus $\mathcal{L}%
_{n1}^{\left(  2,1\right)  }=\mathcal{B}_{1,k}+o_{\mathbf{P}}\left(  1\right)
,$ where%
\begin{equation}
\mathcal{B}_{1,k}:=-\gamma_{1}\left(  1-p\right)  \left(  D_{\gamma}\right)
_{\ast}\beta_{\ast}C^{-\gamma\beta_{\ast}}\int_{0}^{1}t^{\gamma_{1}\beta
_{\ast}}K\left(  t\right)  dt\left\{  \sqrt{k}\left(  k/n\right)
^{\gamma\beta_{\ast}}\right\}  . \label{Bias-1}%
\end{equation}
Once again, making use of Lemma $\ref{lemma1},$ we show that
\[
\mathcal{L}_{n1}^{\left(  2,2\right)  }=O_{\mathbf{P}}\left(  1\right)
k^{-1/2+p}\left(  k/n\right)  ^{\gamma\beta_{\ast}}\left\{  k^{-1}\sum
_{i=1}^{k}u_{i,k}^{-1+\gamma\beta_{\ast}}\right\}  ,
\]
as $n\rightarrow\infty.$ Since $k^{-1}\sum_{i=1}^{k}u_{i,k}^{-1+\gamma
\beta_{\ast}}$ converges to $\int_{0}^{1}s^{-1+\gamma\beta_{\ast}}%
ds=1/\gamma\beta_{\ast}$ (Riemann sum) and both $k^{-1/2+p}$ and $\left(
k/n\right)  ^{\gamma\beta_{\ast}}$ tend to zero, this means that
$\mathcal{L}_{n1}^{\left(  2,2\right)  }=o_{\mathbf{P}}\left(  1\right)  ,$
thus $\mathcal{L}_{n1}^{\left(  2\right)  }=\mathcal{B}_{1,k}+o_{\mathbf{P}%
}\left(  1\right)  .$ To finish with the term $\mathcal{L}_{n1},$ we will also
show that $\sqrt{k}\mathcal{L}_{n1}^{\left(  3\right)  }=o_{\mathbf{P}}\left(
1\right)  $ as $n\rightarrow\infty.$ Recall that $g_{K}^{\prime}$ is a
bounded, then%
\begin{align*}
\mathcal{L}_{n1}^{\left(  3\right)  }  &  =O_{\mathbf{P}}\left(  1\right)
\frac{\sqrt{k}}{k+1}\sum_{i=2}^{k}\left\{  u_{i,k}^{-1}\sum_{j=2}^{i}%
u_{j,k}^{p-1}\right\} \\
&  \times\left\vert \mathbf{1}\left\{  U_{n-i+1}\leq p\circ U_{H}\left(
Y_{n-i+1:n}\right)  \right\}  -\mathbf{1}\left\{  U_{n-i+1}\leq p\circ
U_{H}\left(  n/i\right)  \right\}  \right\vert .
\end{align*}
By using similar arguments as used for the term $B_{1,k}^{\left(  1\right)  }$
in \cite{BWW2019}, we show that $\sqrt{k}\mathcal{L}_{n1}^{\left(  3\right)
}=o_{\mathbf{P}}\left(  1\right)  ,$ therefore we omit details. We now
consider the term%
\[
\sum_{i=2}^{k}B_{i,n}S_{i,k}=\frac{1}{\gamma_{1}}b_{n,k}\sum_{i=2}^{k}%
u_{i,k}^{\beta_{\ast}\gamma}\left(  \frac{1}{i}\sum_{j=2}^{i}RF_{j}%
g_{K}^{\prime}\left(  RF_{j}\right)  \frac{\xi_{j}^{\prime}}{j}\right)
E_{i}^{\left(  n\right)  }.
\]
By using similar decomposition as used to the term $\sum_{i=2}^{k}%
A_{i,n}S_{i,k},$ we end up with
\begin{align*}
\sum_{i=2}^{k}B_{i,n}S_{i,k}  &  =\left(  1+o_{\mathbf{P}}\left(  1\right)
\right)  \frac{1}{\gamma_{1}}b_{n,k}\frac{1}{k+1}\sum_{i=2}^{k}u_{i,k}%
^{\beta_{\ast}\gamma}\left(  \frac{1}{i}\sum_{j=2}^{i}u_{j,k}^{p-1}%
g_{K}^{\prime}\left(  u_{j,k}^{p}\right)  \xi_{j}^{\prime}\right)  E_{i}\\
&  =\left(  1+o_{\mathbf{P}}\left(  1\right)  \right)  pb_{n,k}\left\{
\frac{1}{k+1}\sum_{i=2}^{k}u_{i,k}^{\beta_{\ast}\gamma}\left(  \frac{1}{i}%
\sum_{j=2}^{i}u_{j,k}^{p-1}g_{K}^{\prime}\left(  u_{j,k}^{p}\right)  \right)
E_{i}\right\}  .
\end{align*}
From Lemma $\ref{lemma2}$ the previous factor between two braces converges in
probability, as $n\rightarrow\infty,$ to $p^{-2}\int_{0}^{1}t^{\gamma_{1}%
\beta_{\ast}}K\left(  t\right)  dt,$ therefore $\sqrt{k}\sum_{i=2}^{k}%
B_{i,n}S_{i,k}=\mathcal{B}_{2,k}+o_{\mathbf{P}}\left(  1\right)  ,$ where%
\begin{equation}
\mathcal{B}_{2,k}:=-\gamma^{2}\beta_{\ast}D_{\ast}C^{-\gamma\beta_{\ast}%
}p^{-1}\int_{0}^{1}t^{\gamma_{1}\beta_{\ast}}K\left(  t\right)  dt\left\{
\sqrt{k}\left(  k/n\right)  ^{\gamma\beta_{\ast}}\right\}  . \label{Bais-2}%
\end{equation}
In conclusion, we showed that $\sqrt{k}\mathcal{T}_{k,n}^{\left(
1,1,1\right)  }=\mathcal{N}\left(  0,\sigma_{K}^{2}\right)  +\mathcal{B}%
_{1,k}+\mathcal{B}_{2,k}+o_{\mathbf{P}}\left(  1\right)  .$ Let us now
consider the term $\mathcal{T}_{k,n}^{\left(  1,1,2\right)  }.$ Following the
same steps as used in the proof of subsection A.3.3 in \cite{BWW2019}, we show
that $\sqrt{k}\mathcal{T}_{k,n}^{\left(  1,1,2\right)  }=\mathcal{B}%
_{3,k}+o_{\mathbf{P}}\left(  1\right)  ,$ where
\begin{equation}
\mathcal{B}_{3,k}:=-\gamma^{2}\beta_{\ast}D_{1}C^{-\gamma\beta_{\ast}}%
p^{-2}\int_{0}^{1}t^{\gamma_{1}\beta_{\ast}}K\left(  t\right)  dt\left\{
\sqrt{k}\left(  k/n\right)  ^{\gamma\beta_{\ast}}\right\}  \mathbf{1}\left\{
\beta_{1}\leq\beta_{2}\right\}  , \label{Bais-3}%
\end{equation}
thereby $\sqrt{k}\mathcal{T}_{k,n}^{\left(  1,1\right)  }=\mathcal{N}\left(
0,\sigma_{K}^{2}\right)  +\mathcal{B}_{1,k}+\mathcal{B}_{2,k}-\mathcal{B}%
_{3,k}+o_{\mathbf{P}}\left(  1\right)  .$ Using similar arguments as used to
the proof given in subsection A.3.2 of the same paper, we also show that
$\sqrt{k}\mathcal{T}_{k,n}^{\left(  1,2\right)  }=\mathcal{B}_{4,k}%
+o_{\mathbf{P}}\left(  1\right)  ,$ where
\begin{equation}
\mathcal{B}_{4,k}:=-\gamma^{2}\beta_{\ast}\left(  D_{\ast}+\frac{D^{\ast}}%
{p}\right)  C^{-\gamma\beta_{\ast}}p^{-1}\int_{0}^{1}t^{\gamma_{1}\beta_{\ast
}}K\left(  t\right)  dt\left\{  \sqrt{k}\left(  \frac{k}{n}\right)
^{\gamma\beta_{\ast}}\right\}  , \label{Bais-4}%
\end{equation}
with $D^{\ast}:=-pD_{\ast}\mathbf{1}\left\{  \beta_{2}<\beta_{1}\right\}
+\left(  D_{1}-pD_{\ast}\right)  \mathbf{1}\left\{  \beta_{1}\leq\beta
_{2}\right\}  .$ To summarize, we showed that%
\begin{equation}
\sqrt{k}\left(  \widehat{\gamma}_{1,k}-\gamma_{1}\right)  =\mathcal{L}%
_{n1}^{\left(  1,1\right)  }+\mathcal{B}_{k}+o_{\mathbf{P}}\left(  1\right)
,\text{ as }n\rightarrow\infty, \label{app-f}%
\end{equation}
where $\mathcal{L}_{n1}^{\left(  1,1\right)  }\overset{D}{\rightarrow
}\mathcal{N}\left(  0,\sigma_{K}^{2}\right)  $ and $\mathcal{B}_{k}%
:=\mathcal{B}_{1,k}+\mathcal{B}_{2,k}-\mathcal{B}_{3,k}+\mathcal{B}_{4,k}.$
Substituting the four biases by their corresponding formulas, we end up with
$\mathcal{B}_{k}=m_{K}\left\{  \sqrt{k}\left(  k/n\right)  ^{\gamma\beta
_{\ast}}\right\}  ,$ where $m_{K}$ is as in Theorem $\ref{Theorem1},$ which
completes the proof.

\subsection{Proof of Theorem $\ref{Theorem2}$}

Note that $\widehat{\tau}_{1}=-\beta_{1}\widehat{\gamma}_{1,k}$ $\ $and
$\rho\left(  -\beta_{1}\widehat{\gamma}_{1,k}\right)  =\widehat{\rho}$ where
$\rho\left(  \cdot\right)  $ is as in $\left(  \ref{rho-t}\right)  ,$ it
follows that
\[
\widehat{\gamma}_{1,k}^{\ast}=\widehat{\gamma}_{1,k}-\rho\left(  -\beta
_{1}\widehat{\gamma}_{1,k}\right)  \left\{  T_{k}\left(  \beta_{1};K\right)
-\widehat{\gamma}_{1,k}\int_{0}^{1}s^{\beta_{1}\widehat{\gamma}_{1,k}}K\left(
s\right)  ds\right\}  .
\]
It is obvious that
\begin{align*}
\widehat{\gamma}_{1,k}^{\ast}-\gamma_{1}  &  =\left(  \widehat{\gamma}%
_{1,k}-\gamma_{1}\right)  -\rho\left(  -\beta_{1}\widehat{\gamma}%
_{1,k}\right)  \left\{  T_{k}\left(  \beta_{1};K\right)  -\gamma_{1}\int
_{0}^{1}s^{\beta_{1}\gamma_{1}}K\left(  s\right)  ds\right\} \\
&  \ \ +\rho\left(  -\beta_{1}\widehat{\gamma}_{1,k}\right)  \left(
\widehat{\gamma}_{1,k}\int_{0}^{1}s^{\beta_{1}\widehat{\gamma}_{1,k}}K\left(
s\right)  ds-\gamma_{1}\int_{0}^{1}s^{\beta_{1}\gamma_{1}}K\left(  s\right)
ds\right)  .
\end{align*}
Using Taylor's expansion to function $t\rightarrow t\int_{0}^{1}s^{t}K\left(
s\right)  ds,$ we get
\begin{align*}
&  \widehat{\gamma}_{1,k}\int_{0}^{1}s^{\beta_{1}\widehat{\gamma}_{1,k}%
}K\left(  s\right)  ds-\gamma_{1}\int_{0}^{1}s^{\beta_{1}\gamma_{1}}K\left(
s\right)  ds\\
&  =\left(  \widehat{\gamma}_{1,k}-\gamma_{1}\right)  \int_{0}^{1}s^{\beta
_{1}\gamma_{1}}\left(  1+\beta_{1}\gamma_{1}\log s\right)  K\left(  s\right)
ds\\
&  +\frac{1}{2}\left(  \widehat{\gamma}_{1,k}-\gamma_{1}\right)  ^{2}\beta
_{1}\int_{0}^{1}s^{\beta_{1}\widetilde{\gamma}_{1,k}}\left(  \log s\right)
\left(  \beta_{1}\widetilde{\gamma}_{1,k}\log s+2\right)  K\left(  s\right)
ds,
\end{align*}
where $\widetilde{\gamma}_{1,k}$ is between $\widehat{\gamma}_{1,k}$ and
$\gamma_{1}.$ Note that $s^{t}\left\vert \log^{m}s\right\vert <\exp\left(
-2m/t\right)  \leq1,$ for any $0<s\leq1,$ $m\geq0$ and $t>0.$ On the other
hand $K$ is bounded on the real line and $\widehat{\gamma}_{1,k}%
\overset{\mathbf{P}}{\rightarrow}\gamma_{1},$ then
\[
\int_{0}^{1}s^{\beta_{1}\widetilde{\gamma}_{1,k}}\left\vert \left(  \log
s\right)  \left(  \beta_{1}\widetilde{\gamma}_{1,k}\log s+2\right)
\right\vert K\left(  s\right)  ds=O_{\mathbf{P}}\left(  1\right)  ,\text{ as
}n\rightarrow\infty.
\]
From Theorem $\ref{Theorem1},$ we deduce that $\left(  \widehat{\gamma}%
_{1,k}-\gamma_{1}\right)  ^{2}=O_{\mathbf{P}}\left(  k^{-1}\right)  ,$
therefore
\[
\widehat{\gamma}_{1,k}\int_{0}^{1}s^{\beta_{1}\widehat{\gamma}_{1,k}}K\left(
s\right)  ds-\gamma_{1}\int_{0}^{1}s^{\beta_{1}\gamma_{1}}K\left(  s\right)
ds=\eta_{1}\left(  \widehat{\gamma}_{1,k}-\gamma_{1}\right)  +O_{\mathbf{P}%
}\left(  k^{-1}\right)  ,
\]
where $\eta_{1}=\eta_{1}\left(  \tau_{1}\right)  $ is as in $\left(
\ref{eta1}\right)  .$ To summarize, we showed that
\begin{align*}
\widehat{\gamma}_{1,k}^{\ast}-\gamma_{1}  &  =\left(  1+\eta_{1}\rho\left(
-\beta_{1}\widehat{\gamma}_{1,k}\right)  \right)  \left(  \widehat{\gamma
}_{1,k}-\gamma_{1}\right) \\
&  -\rho\left(  -\beta_{1}\widehat{\gamma}_{1,k}\right)  \left\{  T_{k}\left(
\beta_{1};K\right)  -\gamma_{1}\int_{0}^{1}s^{\beta_{1}\gamma_{1}}K\left(
s\right)  ds\right\}  +O_{\mathbf{P}}\ \left(  k^{-1}\right)  .
\end{align*}
In the proof of Theorem $\ref{Theorem1}$ (equation $\left(  \ref{app-f}%
\right)  ),$ we stated that%
\[
\sqrt{k}\left(  \widehat{\gamma}_{1,k}-\gamma_{1}\right)  =\gamma\frac
{\sqrt{k}}{k+1}\sum_{i=2}^{k}\left\{  u_{i,k}^{-1}\sum_{j=2}^{i}u_{j,k}%
^{p-1}g_{K}^{\prime}\left(  u_{j,k}^{p}\right)  \right\}  A_{i,n}%
+\mathcal{B}_{k}+o_{\mathbf{P}}\left(  1\right)  ,
\]
where $A_{i,n}:=p\left(  E_{i}-1\right)  -\left(  \mathbf{1}\left\{  U_{i}\leq
p\right\}  -p\right)  $ and $\mathcal{B}_{k}:=m_{K}\sqrt{k}\left(  k/n\right)
^{\gamma\beta_{1}},$ where $m_{K}:=-\mathbf{1}\left\{  \beta_{1}\leq\beta
_{2}\right\}  \beta_{1}D_{1}C^{-\gamma\beta_{1}}\gamma_{1}^{2}\eta_{2},$ with
$\eta_{2}=\eta_{2}\left(  \tau_{1}\right)  $ is as in $\left(  \ref{eta12}%
\right)  .$ Next we provide a Gaussian approximation to $T_{k}\left(
\beta_{1};K\right)  $ as well. To this end, we will follow similar steps as
used in the proof of Theorem $\ref{Theorem1}$ as well as that of Theorem 1 in
\cite{BWW2019}. Let us write
\[
\left(  \frac{Z_{n-j:n}}{Z_{n-k:n}}\right)  ^{-\beta_{1}}-\left(
\frac{Z_{n-j+1:n}}{Z_{n-k:n}}\right)  ^{-\beta_{1}}=\exp\left(  -\beta_{1}%
\log\frac{Z_{n-j:n}}{Z_{n-k:n}}\right)  -\exp\left(  -\beta_{1}\log
\frac{Z_{n-j+1:n}}{Z_{n-k:n}}\right)  .
\]
Once again, by using Taylor's expansion to function $t\rightarrow\exp\left(
-\beta_{1}t\right)  ,$ yields
\begin{align*}
&  \frac{1}{\beta_{1}}\left\{  \left(  \frac{Z_{n-j:n}}{Z_{n-k:n}}\right)
^{-\beta_{1}}-\left(  \frac{Z_{n-j+1:n}}{Z_{n-k:n}}\right)  ^{-\beta_{1}%
}\right\} \\
&  =\left(  \frac{Z_{n-j+1:n}}{Z_{n-kn}}\right)  ^{-\beta_{1}}\log
\frac{Z_{n-j+1:n}}{Z_{n-j:n}}+\frac{\beta_{1}}{2}\left(  \log\frac
{Z_{n-j+1:n}}{Z_{n-j:n}}\right)  ^{2}\left(  \frac{\widetilde{Z}_{j:n}%
}{Z_{n-k:n}}\right)  ^{-\beta_{1}},
\end{align*}
for some rv's $\widetilde{Z}_{j:n}$ satisfying $Z_{n-j:n}\leq\widetilde
{Z}_{j:n}\leq$ $Z_{n-j+1:n}.$ Thus $T_{k}\left(  \beta_{1};K\right)  $ may be
decomposed into the sum of%
\[
\widetilde{T}_{k}\left(  \beta_{1};K\right)  :=\sum_{j=2}^{k}\frac
{\overline{F}_{n}^{KM}\left(  Z_{n-j+1:n}\right)  }{\overline{F}_{n}%
^{KM}\left(  Z_{n-k:n}\right)  }K\left(  \frac{\overline{F}_{n}^{KM}\left(
Z_{n-j+1:n}\right)  }{\overline{F}_{n}^{KM}\left(  Z_{n-k:n}\right)  }\right)
\left(  \frac{Z_{n-j+1:n}}{Z_{n-kn}}\right)  ^{-\beta_{1}}\log\frac
{Z_{n-j+1:n}}{Z_{n-j:n}}%
\]
and%
\[
R_{n}:=\frac{\beta_{1}}{2}\sum_{j=2}^{k}\frac{\overline{F}_{n}^{KM}\left(
Z_{n-j+1:n}\right)  }{\overline{F}_{n}^{KM}\left(  Z_{n-k:n}\right)  }K\left(
\frac{\overline{F}_{n}^{KM}\left(  Z_{n-j+1:n}\right)  }{\overline{F}_{n}%
^{KM}\left(  Z_{n-k:n}\right)  }\right)  \left(  \frac{\widetilde{Z}_{j:n}%
}{Z_{n-k:n}}\right)  ^{-\beta_{1}}\left(  \log\frac{Z_{n-j+1:n}}{Z_{n-j:n}%
}\right)  ^{2}.
\]
Since $K$ is bounded on $\mathbb{R},$ then $R_{n}=O_{\mathbf{P}}\left(
\widetilde{R}_{n}\right)  ,$ where
\[
\widetilde{R}_{n}:=\frac{\beta_{1}}{2}\sum_{j=2}^{k}\frac{\overline{F}%
_{n}^{KM}\left(  Z_{n-j+1:n}\right)  }{\overline{F}_{n}^{KM}\left(
Z_{n-k:n}\right)  }\left(  \frac{\widetilde{Z}_{j:n}}{Z_{n-k:n}}\right)
^{-\beta_{1}}\left(  \log\frac{Z_{n-j+1:n}}{Z_{n-j:n}}\right)  ^{2}.
\]
The remainder term $\widetilde{R}_{n}$ corresponds to $R_{n}^{\left(
1\right)  }$ stated in equation $\left(  22\right)  $ in \cite{BWW2019} which
is negligible in the sense that $\sqrt{k}R_{n}^{\left(  1\right)
}=o_{\mathbf{P}}\left(  1\right)  $ thus $\sqrt{k}R_{n}=o_{\mathbf{P}}\left(
1\right)  $ as well. We now focus on the statistic $\widetilde{T}_{k}\left(
\beta_{1};K\right)  $ which is somewhat similar to the kernel estimator
$\widehat{\gamma}_{1,k}.$ Then using similar arguments as used in the proof of
Theorem $\ref{Theorem1}$ we provide a Gaussian approximation to this one
without further details. We summarize the result as follows%
\begin{align*}
&  \sqrt{k}\left\{  T_{k}\left(  \beta_{1};K\right)  -\gamma_{1}\int_{0}%
^{1}s^{\beta_{1}\gamma_{1}}K\left(  s\right)  ds\right\} \\
&  =\gamma\frac{\sqrt{k}}{k+1}\sum_{i=2}^{k}\left\{  u_{i,k}^{-1}\sum
_{j=2}^{i}u_{j,k}^{p-1}g_{K}^{\prime\ast}\left(  u_{j,k}^{p}\right)  \right\}
A_{i,n}+\mathbb{B}_{k}+o_{\mathbf{P}}\left(  1\right)  ,
\end{align*}
where $g_{K}^{\ast}\left(  s\right)  :=s^{\beta_{1}\gamma_{1}+1}K\left(
s\right)  $ and $\mathbb{B}_{k}:=-\mathbf{1}\left\{  \beta_{1}\leq\beta
_{2}\right\}  \beta_{1}D_{1}C^{-\gamma\beta_{1}}\gamma_{1}^{2}\eta_{3}\sqrt
{k}\left(  k/n\right)  ^{\gamma\beta_{1}},$ with $\eta_{3}=\eta_{3}\left(
\tau_{1}\right)  $ is as in $\left(  \ref{eta12}\right)  .$ Recall that
$\widehat{\tau}_{1}=-\beta_{1}\widehat{\gamma}_{1,k}$ is a consistent
estimator for $\tau_{1},$ then by means of the convergence dominate theorem,
we easily showed that $\rho\left(  \widehat{\tau}_{1}\right)  \overset
{\mathbf{P}}{\rightarrow}\rho\left(  \tau_{1}\right)  =\rho\left(  -\beta
_{1}\gamma_{1}\right)  ,$ as $n\rightarrow\infty,$ where
\begin{equation}
\rho\left(  -\beta_{1}\gamma_{1}\right)  =\rho\left(  \tau_{1}\right)
=\frac{\eta_{2}}{\eta_{3}-\eta_{1}\eta_{2}}. \label{rho}%
\end{equation}
Thus, in view of the above two Gaussian approximations, we get%
\begin{equation}
\sqrt{k}\left(  \widehat{\gamma}_{1,k}^{\ast}-\gamma_{1}\right)  =\gamma
\frac{\sqrt{k}}{k+1}\sum_{i=2}^{k}\left\{  u_{i,k}^{-1}\sum_{j=2}^{i}%
u_{j,k}^{p-1}\varphi\left(  u_{j,k}^{p}\right)  \right\}  A_{i,n}%
+\mathbb{B}_{k}^{\ast}+o_{\mathbf{P}}\left(  1\right)  , \label{GA}%
\end{equation}
where
\begin{equation}
\varphi\left(  s\right)  :=\left(  1+\eta_{1}\rho\left(  -\beta_{1}\gamma
_{1}\right)  \right)  g_{K}^{\prime}\left(  s\right)  -\rho\left(  -\beta
_{1}\gamma_{1}\right)  g_{K}^{\prime\ast}\left(  s\right)  , \label{psin}%
\end{equation}
and $\mathbb{B}_{k}^{\ast}:=\left(  1+\eta_{1}\rho\left(  -\beta_{1}\gamma
_{1}\right)  \right)  \mathcal{B}_{k}-\rho\left(  -\beta_{1}\gamma_{1}\right)
\mathbb{B}_{k}.$ The objective now is to establish the asymptotic normality of
$\widehat{\gamma}_{1,k}^{\ast}.$ Using similar arguments as used to the term
$\mathcal{L}_{n1}^{\left(  1\right)  }$ in the proof of Theorem
$\ref{Theorem1},$ we also show that the first term in $\left(  \ref{GA}%
\right)  ,$ converges in distribution to $\mathcal{N}\left(  0,\sigma
_{K}^{\ast2}\right)  ,$ where $\sigma_{K}^{\ast2}:=p\gamma^{2}\int_{0}%
^{1}\left(  s^{-1}\int_{0}^{s}t^{p-1}\varphi\left(  t^{p}\right)  dt\right)
^{2}ds.$ Using elementary algebra, we obtain%
\[
\sigma_{K}^{\ast2}=p\gamma_{1}^{2}\int_{0}^{1}t^{-1/p+1}\left(  \left(
1+\eta_{1}\rho\left(  \tau_{1}\right)  \right)  -\rho\left(  \tau_{1}\right)
t^{-\tau_{1}}\right)  ^{2}K^{2}\left(  t\right)  dt,
\]
which meets the asymptotic variance stated in Theorem $\ref{Theorem2}.$ The
explicit form of the bias term $\mathbb{B}_{k}^{\ast}$ is
\[
\left\{  -\mathbf{1}\left\{  \beta_{1}\leq\beta_{2}\right\}  \beta_{1}%
D_{1}C^{-\gamma\beta_{1}}\gamma_{1}^{2}\right\}  \left\{  \left(  1+\eta
_{1}\rho\left(  \tau_{1}\right)  \right)  \eta_{2}-\rho\left(  \tau
_{1}\right)  \eta_{3}\right\}  \sqrt{k}\left(  k/n\right)  ^{\gamma\beta_{1}%
}.
\]
Substituting $\rho\left(  \tau_{1}\right)  $ by its expression $\left(
\ref{rho}\right)  ,$ we get $\left(  1+\eta_{1}\rho\left(  \tau_{1}\right)
\right)  \eta_{2}-\rho\left(  \tau_{1}\right)  \eta_{3}=0,$ it follows that
$\sqrt{k}\left(  \widehat{\gamma}_{1,k}^{\ast}-\gamma_{1}\right)
\overset{\mathcal{D}}{=}\mathcal{N}\left(  0,\sigma_{K}^{\ast2}\right)
+o_{\mathbf{P}}\left(  1\right)  ,$ which completes the proof the
Theorem.\medskip

\noindent\textbf{Conclusion. }We proposed a smoothed (or a kernel) version of
Worms's estimator \citep{WW2014} of the tail index of a Pareto-type
distribution for randomly censored data. This estimator is a generalization of
the well-known kernel estimator of the extreme value index for complete data
introduced by \cite{CDM85}. The corresponding bias-reduced version of the new
kernel estimator is derived and its asymptotic normality is established. One
of the main features of this estimator is its stability along the interval of
the number $k$ of top extreme values, contrary to Worms's one which behaves
erratically in a zig-zag way. The simulation study showed that, in the case of
weak censoring, the given estimator overall performed better than the
non-smoothed one in terms of bias and MSE. However, in the case of strong
censoring, the MSE of Worms's estimator seems to be better.

\section{\textbf{Appendix}}

\begin{proposition}
\textbf{\label{Propo1}}We have%
\[
\overline{F}_{n}^{KM}\left(  Z_{n-j:n}\right)  -\overline{F}_{n}^{KM}\left(
Z_{n-j+1:n}\right)  =\frac{\delta_{\left(  n-j+1\right)  }}{n\overline{G}%
_{n}^{KM}\left(  Z_{n-j:n}\right)  },\text{ for }j=1,...,k.
\]

\end{proposition}

\begin{proof}
It is clear that%
\begin{align*}
\overline{F}_{n}^{KM}\left(  Z_{n-j:n}\right)  -\overline{F}_{n}^{KM}\left(
Z_{n-j+1:n}\right)   &  =\prod\limits_{i=1}^{n-j}\left(  1-\dfrac
{\delta_{\left(  i\right)  }}{n-i+1}\right)  -\prod\limits_{i=1}%
^{n-j+1}\left(  1-\dfrac{\delta_{\left(  i\right)  }}{n-i+1}\right) \\
&  =\frac{\delta_{\left(  n-j+1\right)  }}{j}\prod\limits_{i=1}^{n-j}\left(
1-\dfrac{\delta_{\left(  i\right)  }}{n-i+1}\right)  ,
\end{align*}
thus%
\begin{equation}
\overline{F}_{n}^{KM}\left(  Z_{n-j:n}\right)  -\overline{F}_{n}^{KM}\left(
Z_{n-j+1:n}\right)  =\frac{\delta_{\left(  n-j+1\right)  }}{j}\overline{F}%
_{n}^{KM}\left(  Z_{n-j:n}\right)  . \label{D}%
\end{equation}
Let $H_{n}\left(  z\right)  :=n^{-1}\sum_{j=1}^{n}\mathbf{1}\left\{  Z_{j}\leq
z\right\}  $ be the empirical cdf pertaining to the sample $Z_{1},...,Z_{n}.$
Since $\overline{H}_{n}\left(  Z_{n-j:n}\right)  =j/n$ then $\overline{F}%
_{n}^{KM}\left(  Z_{n-j:n}\right)  -\overline{F}_{n}^{KM}\left(
Z_{n-j+1:n}\right)  $ equals%
\[
\frac{\delta_{\left(  n-j+1\right)  }}{n}\frac{\overline{F}_{n}^{KM}\left(
Z_{n-j:n}\right)  }{\overline{H}_{n}\left(  Z_{n-j:n}\right)  }=\frac
{\delta_{\left(  n-j+1\right)  }}{n}\frac{\overline{F}_{n}^{KM}\left(
Z_{n-j+1:n}^{-}\right)  }{\overline{H}_{n}\left(  Z_{n-j+1:n}^{-}\right)  }.
\]
From assertion $\left(  11\right)  $ in \cite{SW86} (page 295), we infer that
\[
\frac{\overline{F}_{n}^{KM}\left(  z^{-}\right)  }{\overline{H}_{n}\left(
z^{-}\right)  }=\frac{1}{\overline{G}_{n}^{KM}\left(  z^{-}\right)  },
\]
therefore
\[
\overline{F}_{n}^{KM}\left(  Z_{n-j:n}\right)  -\overline{F}_{n}^{KM}\left(
Z_{n-j+1:n}\right)  =\frac{\delta_{\left(  n-j+1\right)  }}{n\overline{G}%
_{n}^{KM}\left(  Z_{n-j+1:n}^{-}\right)  }=\frac{\delta_{\left(  n-j+1\right)
}}{n\overline{G}_{n}^{KM}\left(  Z_{n-j:n}\right)  },
\]
as sought.
\end{proof}

\begin{proposition}
\textbf{\label{Propo2}}We have\textbf{\ }%
\[
\frac{\overline{F}_{n}^{KM}\left(  Z_{n-j:n}\right)  }{\overline{F}_{n}%
^{KM}\left(  Z_{n-k:n}\right)  }-\frac{\overline{F}_{n}^{KM}\left(
Z_{n-j+1:n}\right)  }{\overline{F}_{n}^{KM}\left(  Z_{n-k:n}\right)
}=O_{\mathbf{P}}\left(  k^{-p}\right)  ,
\]
uniformly on $1\leq j\leq k,$ which tends to zero in probability as
$n\rightarrow\infty.$
\end{proposition}

\begin{proof}
In view of statement $\left(  \ref{D}\right)  ,$ we write%
\begin{equation}
\frac{\overline{F}_{n}^{KM}\left(  Z_{n-j:n}\right)  }{\overline{F}_{n}%
^{KM}\left(  Z_{n-k:n}\right)  }-\frac{\overline{F}_{n}^{KM}\left(
Z_{n-j+1:n}\right)  }{\overline{F}_{n}^{KM}\left(  Z_{n-k:n}\right)  }%
=\frac{\delta_{\left(  n-j+1\right)  }}{j}\frac{\overline{F}_{n}^{KM}\left(
Z_{n-j:n}\right)  }{\overline{F}_{n}^{KM}\left(  Z_{n-k:n}\right)  }.
\label{RA}%
\end{equation}
Indeed, form \cite{Gill80} (page 39) and \cite{Zhou91} (Theorem 2.2) we have
\[
\frac{\overline{F}_{n}^{KM}\left(  x\right)  }{\overline{F}\left(  x\right)
}=O_{\mathbf{P}}\left(  1\right)  =\frac{\overline{F}\left(  x\right)
}{\overline{F}_{n}^{KM}\left(  x\right)  },\text{ }%
\]
uniformly on $x<Z_{n:n}.$ This implies that the right-side of equation
$\left(  \ref{RA}\right)  $ is stochastically bounded by $j^{-1}\overline
{F}\left(  Z_{n-j:n}\right)  /\overline{F}\left(  Z_{n-k:n}\right)  ,$
uniformly over $1\leq j\leq k.$ Recall that $\overline{F}\in\mathcal{RV}%
_{\left(  -1/\gamma_{1}\right)  },$ then from Potter's inequalities B.1.19 in
\cite{deHF06} (page 367) we have: for $\epsilon>0$ and $x\geq1,$ there exists
$t_{0}>0,$ such that for $t\geq t_{0},$ $\left(  1-\epsilon\right)
x^{-1/\gamma_{1}-\epsilon}<\overline{F}\left(  tx\right)  /\overline{F}\left(
t\right)  <\left(  1+\epsilon\right)  x^{-1/\gamma_{1}+\epsilon}.$ Letting
$t=Z_{n-k:n}$ and $x=Z_{n-j:n}/Z_{n-k:n}$ and applying these inequalities,
yields%
\begin{equation}
\left(  1-\epsilon\right)  \left(  \frac{Z_{n-j:n}}{Z_{n-k:n}}\right)
^{-1/\gamma_{1}-\epsilon}<\frac{\overline{F}\left(  Z_{n-j:n}\right)
}{\overline{F}\left(  Z_{n-k:n}\right)  }<\left(  1+\epsilon\right)  \left(
\frac{Z_{n-j:n}}{Z_{n-k:n}}\right)  ^{-1/\gamma_{1}+\epsilon}. \label{Finq}%
\end{equation}
Since $U_{H}\in\mathcal{RV}_{\left(  \gamma\right)  },$ then
\[
\left(  1-\epsilon\right)  \left(  \frac{Y_{n-j:n}}{Y_{n-k:n}}\right)
^{\gamma-\epsilon}<\frac{U_{H}\left(  Y_{n-j:n}\right)  }{U_{H}\left(
Y_{n-k:n}\right)  }<\left(  1+\epsilon\right)  \left(  \frac{Y_{n-j:n}%
}{Y_{n-k:n}}\right)  ^{\gamma+\epsilon},
\]
where $Y_{1:n}\leq...\leq Y_{n:n}$ are the order statistics already defined in
the beginning of the proof of Theorem $\ref{Theorem1}.$ On the other hand,
from Corollary 2.2.2 in \cite{deHF06} (page 41), we infer that $\left(
j/n\right)  Y_{n-j:n}\overset{\mathbf{P}}{\rightarrow}1,$ $j=1,...,k,$ as
$n\rightarrow\infty,$ thus by using the previous inequalities, we get%
\[
\frac{Z_{n-j:n}}{Z_{n-k:n}}=\left(  1+o_{\mathbf{P}}\left(  1\right)  \right)
\left(  \frac{j}{k}\right)  ^{-\gamma},\text{ as }n\rightarrow\infty.
\]
Therefore, thanks of $\left(  \ref{Finq}\right)  ,$ we have%
\[
\frac{\overline{F}\left(  Z_{n-j:n}\right)  }{\overline{F}\left(
Z_{n-k:n}\right)  }=\left(  1+o_{\mathbf{P}}\left(  1\right)  \right)  \left(
\frac{j}{k}\right)  ^{p},\text{ }\left(  \text{since }p=\gamma/\gamma
_{1}\right)  ,
\]
uniformly over $1\leq j\leq k.$ Then we showed that%
\[
\frac{\overline{F}_{n}^{KM}\left(  Z_{n-j:n}\right)  }{\overline{F}_{n}%
^{KM}\left(  Z_{n-k:n}\right)  }-\frac{\overline{F}_{n}^{KM}\left(
Z_{n-j+1:n}\right)  }{\overline{F}_{n}^{KM}\left(  Z_{n-k:n}\right)  }=\left(
1+o_{\mathbf{P}}\left(  1\right)  \right)  \frac{1}{j}\left(  \frac{j}%
{k}\right)  ^{p},
\]
which completes the proof, since $\frac{1}{j}\left(  \frac{j}{k}\right)
^{p}=k^{-p}j^{p-1}<k^{-p},$ for all $1\leq j\leq k$ and $0<p<1.$
\end{proof}

\begin{proposition}
\textbf{\label{Propo3}}Let $\left(  a_{j}\right)  _{0\leq j\leq m}$ and
$\left(  b_{j}\right)  _{0\leq j\leq m}$ be two sequences of real numbers such
that $a_{0}=b_{m}=0,$ then $\sum_{j=1}^{m}\left(  a_{j}-a_{j-1}\right)
b_{j-1}=\sum_{j=1}^{m}a_{j}\left(  b_{j-1}-b_{j}\right)  .$
\end{proposition}

\begin{proof}
It is straightforward by elementary algebra.
\end{proof}

\begin{lemma}
\textbf{\label{lemma1}}There exists a positive constant $C=C\left(  K\right)
, $ such that
\[
\left\vert d_{i,k}\left(  g_{K}\right)  \right\vert \leq Cu_{i,k}^{-1}\left(
k+1\right)  ^{-p},\text{ for all }2\leq i\leq k.
\]

\end{lemma}

\begin{proof}
Let us write $d_{i,k}\left(  g_{K}\right)  =u_{i.k}^{-1}\Delta_{i,k}\left(
g_{K}\right)  ,$ where%
\[
\Delta_{i,k}\left(  g_{K}\right)  :=\frac{1}{k+1}\sum_{j=2}^{i}u_{j,k}%
^{p-1}g_{K}^{\prime}\left(  u_{j,k}^{p}\right)  -\int_{0}^{u_{i,k}}%
s^{p-1}g_{K}^{\prime}\left(  s^{p}\right)  ds.
\]
It is easy to check that%
\begin{equation}
\Delta_{i,k}\left(  g_{K}\right)  =\sum_{j=2}^{i}\int_{u_{j-1,k}}^{u_{j,k}%
}\left(  u_{j,k}^{p-1}g_{K}^{\prime}\left(  u_{j,k}^{p}\right)  -s^{p-1}%
g_{K}^{\prime}\left(  s^{p}\right)  \right)  ds-\int_{0}^{u_{1,k}}s^{p-1}%
g_{K}^{\prime}\left(  s^{p}\right)  ds. \label{delta}%
\end{equation}
For convenience, we set $h\left(  s\right)  :=sg_{K}^{\prime}\left(
s^{\frac{p}{p-1}}\right)  ,$ $0<s<1,$ and applying the mean value theorem
yields $h\left(  u_{j,k}^{p-1}\right)  -h\left(  s^{p-1}\right)  =\left(
u_{j,k}^{p-1}-s^{p-1}\right)  h^{\prime}\left(  s_{j}^{p-1}\right)  ,$ where
$s<s_{j}<u_{j,k},$ with $h^{\prime}\left(  s\right)  =g_{K}^{\prime}\left(
s^{\frac{p}{p-1}}\right)  +\frac{p}{p-1}s^{\frac{p}{p-1}}g_{K}^{\prime\prime
}\left(  s^{\frac{p}{p-1}}\right)  .$ On the other terms, we have
\begin{equation}
u_{j,k}^{p-1}g_{K}^{\prime}\left(  u_{j,k}^{p}\right)  -s^{p-1}g_{K}^{\prime
}\left(  s^{p}\right)  =\left(  u_{j,k}^{p-1}-s^{p-1}\right)  h^{\prime
}\left(  s_{j}^{p-1}\right)  , \label{mvt}%
\end{equation}
where $h^{\prime}\left(  s_{j}^{p-1}\right)  =g_{K}^{\prime}\left(  s_{j}%
^{p}\right)  +\frac{p}{p-1}s_{j}^{p}g_{K}^{\prime\prime}\left(  s_{j}%
^{p}\right)  .$ From assumption $\left[  A4\right]  ,$ both $g_{K}^{\prime}$
and $g_{K}^{\prime\prime}$ are bounded, this implies that there exist two
positive constants $M_{1}$ and $M_{2},$ such that $g_{K}^{\prime}\left(
u^{p}\right)  <M_{1}$ and $\left\vert h^{\prime}\left(  u^{p-1}\right)
\right\vert <M_{2},$ for all $0<u<1.$ Thus, combining $\left(  \ref{delta}%
\right)  $ and $\left(  \ref{mvt}\right)  ,$ yields%
\[
\left\vert \Delta_{i,k}\left(  g_{K}\right)  \right\vert \leq M\left\{
\widetilde{\Delta}_{i,k}+p^{-1}u_{1,k}^{p}\right\}  ,
\]
where $\widetilde{\Delta}_{i,k}:=\sum_{j=2}^{i}\int_{u_{j-1,k}}^{u_{j,k}%
}\left\vert u_{j,k}^{p-1}-s^{p-1}\right\vert ds$ and $M:=\max\left(
M_{1},M_{2}\right)  $ that depends on $K.$ Observe now that $\widetilde
{\Delta}_{i,k}=-\sum_{j=2}^{i}\int_{u_{j-1,k}}^{u_{j,k}}\left(  u_{j,k}%
^{p-1}-s^{p-1}\right)  ds=-\Delta_{i,k}+p^{-1}u_{1,k}^{p},$ where
\[
\Delta_{i,k}:=\frac{1}{k+1}\sum_{j=2}^{i}u_{j,k}^{p-1}-\frac{u_{i,k}^{p}}{p}.
\]
From assertion $(38)$ of Lemma 1 in \cite{BWW2019}, we infer that there exists
a positive constants $C^{\ast}$ such that $\Delta_{i,k}\leq C^{\ast}\left(
k+1\right)  ^{-p},$ for all $2\leq i\leq k,$ therefore%
\[
\left\vert \Delta_{i,k}\left(  g_{K}\right)  \right\vert \leq M\left\{
\frac{C^{\ast}}{\left(  k+1\right)  ^{p}}+\frac{2p^{-1}}{\left(  k+1\right)
^{p}}\right\}  =\frac{C}{\left(  k+1\right)  ^{p}},
\]
where $C:=M\left(  C^{\ast}+2p^{-1}\right)  ,$ thus $\left\vert d_{i,k}\left(
g_{K}\right)  \right\vert \leq Cu_{i,k}^{-1}\left(  k+1\right)  ^{-p}$ as sought.
\end{proof}

\begin{lemma}
\textbf{\label{lemma2}}We have%
\[
\pi_{k}:=\frac{1}{k+1}\sum_{i=2}^{k}u_{i,k}^{\beta_{\ast}\gamma}\left(
\frac{1}{i}\sum_{j=2}^{i}u_{j,k}^{p-1}g_{K}^{\prime}\left(  u_{j,k}%
^{p}\right)  \right)  E_{i}\overset{\mathbf{P}}{\rightarrow}\frac{1}{p^{2}%
}\int_{0}^{1}t^{\gamma_{1}\beta_{\ast}}K\left(  t\right)  dt.
\]

\end{lemma}

\begin{proof}
Observe that $\pi_{k}$ may be decomposed into the sum of
\[
\pi_{k,1}:=\frac{1}{k+1}\sum_{i=2}^{k}d_{i,k}\left(  g_{K}\right)
u_{i,k}^{\beta_{\ast}\gamma}E_{i}%
\]
and%
\[
\pi_{k,2}:=\frac{1}{k+1}\sum_{i=2}^{k}\left\{  u_{i,k}^{\beta_{\ast}\gamma
-1}\int_{0}^{u_{i,k}}s^{p-1}g_{K}^{\prime}\left(  s^{p}\right)  ds\right\}
E_{i}.
\]
Using Lemma $\ref{lemma1},$ we get $\mathbf{E}\left\vert \pi_{k,1}\right\vert
\leq C\left(  k+1\right)  ^{-p}\frac{1}{k+1}\sum_{i=1}^{k}u_{i,k}^{\beta
_{\ast}\gamma-1},$ then from the Riemann sum, we have $\frac{1}{k+1}\sum
_{i=1}^{k}u_{i,k}^{\beta_{\ast}\gamma-1}$ tends to $\left(  \beta_{\ast}%
\gamma\right)  ^{-1},$ and since $k^{-p}\rightarrow0,$ then $\pi_{k,1}%
\overset{\mathbf{P}}{\rightarrow}0,$ as $n\rightarrow\infty.$ On the other
hand, by means of Lyapunov's central limit theorem for triangular arrays, we
readily show that $\pi_{k,2}\overset{\mathbf{P}}{\rightarrow}\int_{0}%
^{1}t^{\beta_{\ast}\gamma-1}\left(  \int_{0}^{t}s^{p-1}g_{K}^{\prime}\left(
s^{p}\right)  ds\right)  dt,$ as $n\rightarrow\infty.$ It is easy to verify,
making some change of variables, that this last equals $p^{-2}\int_{0}%
^{1}t^{\gamma_{1}\beta_{\ast}}K\left(  t\right)  dt,$ which completes the proof.
\end{proof}

\end{document}